\documentclass{article}
\usepackage{amsmath,amsthm,amssymb,amscd}

\newcommand{\ga}{\alpha}
\newcommand{\gb}{\beta}
\renewcommand{\gg}{\gamma}
\newcommand{\gd}{\delta}
\newcommand{\gw}{\omega}

\newcommand{\gs}{\sigma}
\newcommand{\eps}{\varepsilon}

\newcommand{\liff}{\leftrightarrow}

\newcommand{\supp}{\mathrm{supp}}

\newcommand{\dom}{\mathrm{dom}}

\newcommand{\rng}{\mathrm{rng}}
\newcommand{\power}{\mathcal{P}}

\newcommand{\id}{\mathrm{id}}

\newcommand{\stab}{\mathrm{stab}}
\newcommand{\pstab}{\mathrm{pstab}}
\newcommand{\acts}{\curvearrowright}

\newtheorem{theorem}{Theorem}[section]

\newtheorem{claim}[theorem]{Claim}

\newtheorem{fact}[theorem]{Fact}
\newtheorem{proposition}[theorem]{Proposition}

\theoremstyle{definition}
\newtheorem{definition}[theorem]{Definition}
\newtheorem{example}[theorem]{Example}

\title{Dynamical ideals and the axiom of choice\footnote{2020 AMS subject classification 03E25, 22F05.}}

\author{
Jind{\v r}ich Zapletal\\
University of Florida\\
zapletal@ufl.edu}

\begin{document}
\maketitle

\begin{abstract}
I provide several natural properties of group actions which translate into fragments of axiom of choice in the associated permutation models of choiceless set theory.
\end{abstract}

\section{Introduction}

The field of symmetric and permutation models of choiceless set theory is commonly perceived as chaotic. One needs to only look at the encyclopedic form of \cite{howard:ac}, complete with a floppy disc, to understand where that impression comes from. The purpose of this paper is to provide several natural properties of group actions which translate into fragments of axiom of choice in the associated permutation models (or Fraenkel--Mostowski models, as they are frequently called). Very many results  quoted in \cite{howard:ac} follow essentially immediately, and every now and then a novel conclusion about a known model appears. The main point is that the properties of group actions considered have intrinsic interest, perhaps justifying the study of permutation models in the eyes of a non-specialist. Evaluation of these properties in very natural cases is often challenging, and it leads to natural questions in model theory, combinatorics, geometric topology, and other fields.

The sections are ordered in decreasing strength of the fragments of axiom of choice they deal with. In Section~\ref{notationsection}, I fix the terminology and gather some well-known or easy results concerning the construction of the permutation models. In Section~\ref{wosection}, I consider the axiom of well-ordered choice, and provide a dynamical equivalent to it, cofinal orbits--Definition~\ref{cofinaldefinition} and Theorem~\ref{cofinaltheorem}. The ideal of nowhere dense subsets of a manifold of any finite dimension has cofinal orbits \cite{young:personal}. For any interesting topological space, the status of cofinal orbits of the ideal of nowhere dense sets seems to be a challenging and often open problem. In Section~\ref{dcsection}, I consider the axiom of dependent choices. A dynamical equivalent to it is isolated in Definition~\ref{dccompletedefinition} and  Theorem~\ref{dctheorem}; its version for symmetric models was considered previously by Karagila and Schilhan \cite{karagila:dc}. In Section~\ref{countablesection}, I consider the axiom of countable choice and its dynamical equivalent, $\gs$-closure. I show that many dynamical ideals have this property, such as the ideal of countable closed subsets of the real line, or the well-ordered subsets of rationals. In Section~\ref{simplesection}, I consider the statement that unions of well-orderable collections of well-orderable sets are well-orderable. There is a neat dynamical criterion similar to topological simplicity of groups. Naturally enough, abelian group actions are never simple except for trivial cases, Theorem~\ref{abeliantheorem}. One simple ideal is the ideal of finite sets on many limit Fraisse structures, Theorem~\ref{fraissetheorem}. Finally, in Section~\ref{finitesection}, I show how to rule out amorphous or infinite, Dedekind-finite sets from permutation models. The criterion uses a stratification of the ideal into an increasing union which exhibits a degree of $\gs$-closure.

\section{Notation and terminology}
\label{notationsection}

In this section, I show how to obtain a permutation model (a Fraenkel--Mostowski model in the terminology of \cite{howard:ac}) of the theory ZF with atoms from an object called a dynamical ideal. The construction is well-known; the section serves mostly to introduce suitable terminology.

\begin{definition}
ZFA, the set theory with atoms, is the theory with the following description.

\begin{enumerate}
\item Its language contains a binary relational symbol $\in$ and a unary relational symbol $\mathbb{A}$ for atoms;
\item its axioms include all usual axioms of ZF, except for the axiom of extensionality, which is stated only for sets which are not atoms;
\item there are two additional axioms: $\forall x\ \mathbb{A}(x)\to\forall y\ y\notin x$ and $\exists y\ \forall x\ y\in x\liff \mathbb{A}(x)$ (atoms form a set).
\end{enumerate}

\noindent ZFCA is ZFA plus the axiom of choice.
\end{definition}

\noindent The main issue in ZFA and ZFCA (about which its axioms say nothing) is the structure of its set of atoms. It is not difficult to build a model of ZFCA with a prescribed set of atoms.

\begin{definition}
\label{vxdefinition}
Let $X$ be a set. There is up to class isomorphism unique class model $M$ of ZFCA satisfying the following demands:

\begin{enumerate}
\item $X=\mathbb{A}^M$;
\item the membership relation of $M$ is well-founded;
\item for every $m\in M$ the collection $\{n\in M\colon n\in^Mm\}$ is a set as opposed to proper class;
\item for every set $B\subset M$ there is an element $m\in M$ such that $B=\{n\in M\colon n\in^Mm\}$.
\end{enumerate}

\noindent The model $M$ will be denoted by $V[[X]]$.
\end{definition}

\noindent An explicit and well-known construction of the model can be found for example in \cite[Lemma 15.47]{jech:newset}. By an abuse of notation, the membership relation of $V[[X]]$ will be denoted by $\in$. By another abuse of notation, I will identify $X$ with the element of the model $V[[X]]$ which contains exactly all the $V[[X]]$-atoms.

\begin{definition}
A set $A\in V[[X]]$ is \emph{pure} if $V[[X]]\models$ the transitive closure of $A$ contains no atoms. The class $\{A\in V[[X]]\colon A$ is pure$\}$ is called the \emph{pure part} of $V[[X]]$ and denoted by $V$.
\end{definition}

\noindent It is immediate that the pure part of $V[[X]]$ is definably isomorphic to the set-theoretic universe in which the model $V[[X]]$ is built. In the construction of inner models of ZFA, group actions play central role. The following definition records key notational elements.

\begin{definition}
Suppose that $\Gamma\acts X$ is a group action.

\begin{enumerate}
\item $\Gamma\acts V[[X]]$ is the unique action extending the original one and satisfying $\gamma\cdot A=\{\gamma\cdot B\colon B\in A\}$ for all $\gamma\in \Gamma$ and $A\in V[[X]]$;
\item if $A\in V[[X]]$ then $\stab(A)$, the stabilizer of $A$, is the set $\{\gamma\in\Gamma\colon \gamma\cdot A=A\}$;
\item if $a\subset X$ then $\pstab(a)$, the \emph{pointwise stabilizer of $x$}, is the set $\{\gamma\in \Gamma\colon\forall x\in a\ \gamma\cdot x=x\}=\bigcap_{x\in a}\stab(x)$.
\end{enumerate}
\end{definition}

\noindent Stabilizers and pointwise stabilizers are subgroups of $\Gamma$. The extension of the action to all of $V[[X]]$ is an action of $\Gamma$ by $\in$-automorphisms; it is not difficult to show by $\in$-recursion that the action $\Gamma\acts V[[X]]$ fixes all pure elements. To construct an inner model of ZFA, an additional piece of data is needed:

\begin{definition}
A \emph{dynamical ideal} is a tuple $\Gamma\acts X, I$ where $\Gamma$ is a group acting on a set $X$ and $I$ is an ideal on the set $X$ containing all singletons, invariant under the action. That is to say, for every $a\in I$ and every $\gamma\in\Gamma$, the set $\gamma\cdot a=\{\gamma\cdot x\colon x\in a\}$ belongs to $I$.
\end{definition}

\noindent One point in this paper is that the class of dynamical ideals is much broader than the examples normally discussed in connection with permutation models, and that there are tools to investigate even the more ``exotic" examples.

\begin{example}
\label{automorphismexample}
Let $\mathcal{S}$ be a structure with countable universe $X$. Let $\Gamma$ be the group of automorphisms of $\mathcal{S}$, acting on $X$ by application. Let $I$ be the ideal of finite subsets of $X$. Then $\Gamma\acts X, I$ is a dynamical ideal.
\end{example}

\begin{example}
\label{homeomorphismexample}
Let $X$ be a topological space, let $\Gamma$ be the homeomorphism group acting on $X$ by application, and let $I$ be an ideal defined from the topology only:

\begin{enumerate}
\item the ideal of nowhere dense sets;
\item the ideal of sets with countable closure;
\item the ideal of sets with zero-dimensional closure.
\end{enumerate}

\noindent Then $\Gamma\acts X, I$ is a dynamical ideal.
\end{example}

\begin{definition}
Let $\Gamma\acts X, I$ be a dynamical ideal. The \emph{associated permutation model} $W[[X]]$ is the transitive part of the class $\{A\in V[[X]]\colon\exists b\in I\colon \pstab(b)\subseteq\stab(A)\}$.
\end{definition}

\noindent Note that the notation $W[[X]]$ abstracts away from the group action and the ideal for the sake of brevity. This should not cause any confusion. The following is well-known \cite[Chapter 4]{jech:choice}.

\begin{fact}
$W[[X]]$ is a model of Zermelo--Fraenkel set theory with atoms.
\end{fact}

In the rest of this section, I derive several properties of the model $W[[X]]$ which hold irrespective of the dynamical ideal used.

\begin{proposition}
\label{siproposition}
Let $\Gamma\acts X, I$ be a dynamical ideal.

\begin{enumerate}
\item $X$ and $I$ both belong to the associated permutation model;
\item $V\subset W[[X]]$;
\item (support invariance) the relation $\{\langle b, A\rangle\colon b\in I$ and $\pstab(b)\subset\stab(A)\}$ is invariant under the group action and as such belongs to the permutation model;
\item the permutation model is invariant under the group action.
\end{enumerate}
\end{proposition}

\begin{proof}
For (1), note that $\{x\}\in I$ and $\pstab(\{x\})=\stab(x)$; in conclusion $X\subset W[[X]]$. Since $X$ is invariant under the action, $X\in W[[X]]$.  For every set $a\in I$, $\pstab(a)\subset\stab(a)$, so $a\in W[[X]]$. Finally, since the ideal $I$ is invariant under the action, $\Gamma=\stab(I)$ holds and $I\in W[[X]]$ as desired. (2) is clear as all pure sets are fixed by the group action.

For (3), suppose that $\pstab(b)\subseteq\stab(A)$ and $\gamma\in\Gamma$ is an element. To show that $\pstab(\gamma\cdot b)\subseteq\stab(\gamma\cdot A)$, let $\gd\in\pstab(\gamma\cdot b)$ be an arbitrary element. Then $\gamma^{-1}\gd\gamma\in\pstab(b)$, so $\gamma^{-1}\gd\gamma\cdot A=A$, and multiplying both sides by $\gamma$, get $\gd\cdot(\gg\cdot A)=\gg\cdot A$. (4) is an immediate corollary of the definition of the permutation model and (3).
\end{proof}

\noindent In the permutation model, we view well-orderable sets as trivial: every structure on a well-orderable set is a copy of a structure in $V$. It will be useful to have a characterization of well-orderable sets.

\begin{proposition}
\label{wotheorem}
Let $\Gamma\acts X, I$ be a dynamical ideal. 
The following are equivalent for every set $A\in W[[X]]$:

\begin{enumerate}
\item $A$ is well-orderable in $W[[X]]$;
\item there is a set $b\in I$ such that $\pstab(b)\subseteq\pstab(A)$;
\item $\power\power(A)\cap W[[X]]=\power\power(A)\cap V[[X]]$;
\item $\power\power(A)$ is well-orderable in $W[[X]]$.
\end{enumerate}
\end{proposition}

\noindent The equivalence of (1) and (4) is exactly the feature of permutation models which sets them apart from most symmetric models of ZF obtained as submodels of generic extensions.

\begin{proof}
For (1) implies (2), let $R$ be a well-ordering of $A$ in the permutation model, let $b\in I$ be such that $\pstab(b)\subseteq\stab(R)$, and by transfinite induction on $R$ prove that $\pstab(b)\subseteq\stab(B)$ for every $B\in A$. For (2) implies (3), note that $\pstab(b)\subset\pstab(B)$ for every $B\subset A$ and then $\pstab(b)\subset\pstab(C)$ for every $C\subset\power(A)\cap V[[X]]$. For (3) implies (1), use the axiom of choice in $V[[X]]$ to find a well-ordering on $A$ and code it as the set of its initial segments (an element of $\power\power(A)$) to transfer it to $W[[X]]$.

Finally, (2) implies (4) since $\pstab(b)\subseteq\pstab(\power\power(A))$ and then $\pstab(b)\subset\stab(R)$ for any relation $R$ on $\power\power(A)$ in $V[[X]]$, in particular for a well-ordering $R$ on $\power\power(A)$ obtained in $V[[X]]$ using the axiom of choice there. (4) implies (1) trivially: the map $B\mapsto \{\{ B\}\}$ is an injection from $A$ to $\power\power(A)$ in $W[[X]]$.
\end{proof}

\noindent Finally, I discuss a natural closure-type property of a dynamical ideal which is necessary for the statement of certain characterization theorems.

\begin{definition}
Let $\Gamma\acts X, I$ be a dynamical ideal.

\begin{enumerate}
\item a set $a\subset X$ is \emph{dynamicaly closed} if for every $x\notin a$ there is $\gamma\in \pstab(a)$ such that $\gamma\cdot x\neq x$;
\item the dynamical ideal is \emph{dynamically closed} if every set in $I$ is a subset of a dynamically closed set in $I$.
\end{enumerate}
\end{definition}

\noindent It is not hard to see that for every set $a\subset X$ there is the smallest set $b\subset X$ which is definably closed and $a\subset b$, namely $b=\{x\in X\colon \pstab(a)\subset\stab(x)\}$; I will call $b$ the \emph{dynamical closure of $a$}. A brief diagram-chasing argument shows that the dynamical closure operator is invariant under the group action. For every dynamical ideal $I$ there is also the smallest dynamical ideal $J$ which is dynamically closed and $I\subset J$, namely the ideal of all sets which are subsets of dynamical closures of sets in $I$. It is easy to check that the ideal $J$ thus defined is invariant under the group action. I will call $J$ the \emph{dynamical closure of $I$}. The following is nearly trivial.

\begin{proposition}
\label{closureproposition}
Let $\Gamma\acts X, I$ be a dynamical ideal, and let $\Gamma\acts X, J$ be its dynamical closure. 

\begin{enumerate}
\item The two dynamical ideals generate the same permutation model $W[[X]]$;
\item $J=\{a\in W[[X]]\colon a\subseteq X$ and $W[[X]]\models a$ is well-orderable$\}$.
\end{enumerate}
\end{proposition}

\begin{proof}
For (1), work in the model $V[[X]]$, let $A$ be any set, and note that there is a set $b\in I$ such that $\pstab(b)\subseteq\stab(A)$ iff there is a set $a\in J$ such that $\pstab(a)\subseteq\stab(A)$. The left-to-right implication follows from the inclusion $I\subseteq J$. For the right-to-left implication, let $a\in J$ be such that $\pstab(a)\subseteq\stab(A)$ and find a set $b\in I$ such that $\pstab(b)\subseteq\pstab(a)$. Clearly, $\pstab(b)\subseteq\stab(A)$ holds. (1) follows.

For (2), suppose that $a\subset X$ is a set in $W[[X]]$. If $a$ is well-orderable, let $\leq$ be a well-ordering on $a$ in the model $W[[X]]$. Let $b\in I$ be a set such that $\pstab(b)\subseteq\stab(\leq)$. By transfinite induction on $\leq$ argue that for every $x\in a$ and every $\gamma\in\pstab(b)$, $\gamma\cdot x=x$. It follows that $a$ is a subset of the dynamical closure of the set $b$, so $a\in J$. If, on the other hand, $a\in J$, then find a set $b\in I$ such that $\pstab(b)\subseteq\pstab(a)$, observe that for any relation $R\in V[[X]]$ on $a$ $\pstab(b)\subseteq\stab(R)$, and conclude that in particular, the model $W[[X]]$ contains every well-ordering on the set $a$ found in the model $V[[X]]$.
\end{proof}

\begin{example}
Let $X$ be a topological space, let $\Gamma$ be the group of self-homeomorphisms of $X$ acting on $X$ by application, and let $a\subset X$ be a set. The topological closure is a subset of the dynamical closure of $a$. This is the reason why in this situation I consider only ideals generated by closed sets. 
\end{example}

\section{Axiom of well-ordered choice}
\label{wosection}

The first common fragment of axiom of choice under consideration in this paper is the following.

\begin{definition}
\cite[Form 40]{howard:ac}
The \emph{axiom of well-ordered choice} is the statement that every well-ordered family of non-empty sets has a choice function.
\end{definition}

\noindent The dynamical counterpart to the well-ordered axiom of choice is the following.

\begin{definition}
\label{cofinaldefinition}
Let $\Gamma\acts X, I$ be a dynamical ideal. The ideal  has \emph{cofinal orbits} if for every $a\in I$ there is $b\in I$ which is $a$-\emph{large}: for every $c\in I$ there is $\gamma\in\pstab(a)$ such that $c\subseteq \gamma\cdot b$.
\end{definition}

\begin{theorem}
\label{cofinaltheorem}
Let $\Gamma\acts X, I$ be a dynamical ideal.

\begin{enumerate}
\item if the ideal  has cofinal orbits then the associated permutation model satisfies the axiom of well-ordered choice;
\item If the ideal $I$ is dynamically closed and the associated permutation model satisfies the axiom of well-ordered choice, then the ideal has cofinal orbits.
\end{enumerate}
\end{theorem}

\begin{proof}
For (1), let $W[[X]]$ be the permutation model, and suppose that $A$ is a well-orderable set of nonempty sets in $W[[X]]$. Use Proposition~\ref{wotheorem} to find a set $a\in I$ such that $\pstab(a)\subset\pstab(A)$. Let $b\in I$ be a set such that the $\pstab(a)$-orbit of $b$ is cofinal in $I$. Now, we claim that every set $B\in A$ contains a set $C\in B$ such that $\pstab(b)\subseteq\stab(C)$. To see this, let $D\in B$ be an arbitrary set, and let $d\in I$ be such that $\pstab(d)\subseteq\stab(D)$. Find a group element $\gamma\in\pstab(a)$ such that $d\subset\gamma\cdot b$ and let $C=\gamma^{-1}\cdot D$; we claim that the set $C$ works as required. First of all, it is clear that $C\in B$ holds since $\gamma$ (and $\gamma^{-1}$) fixes every element of $A$; in particular, it fixes $B$. Second, a diagram chasing argument shows that for every element $\gd\in\pstab(b)$, $\gamma\gd\gamma^{-1}\in\pstab(d)$, so $\gamma\gd\gamma^{-1}\cdot D=D$ and $\gd\cdot (\gamma^{-1}\cdot D)=\gamma^{-1}\cdot D$ as required.

Now, let $f$ be any selector on $A$ such that for every $B\in A$, $\pstab(b)\subseteq\stab(f(A))$ holds; such a selector exists by the previous paragraph. It is immediate that $\pstab(b)\subseteq\stab(f)$, so $f\in W[[X]]$ and (1) follows.

For (2), suppose that the dynamical ideal is dynamically closed and $W[[X]]$ satisfies well-ordered choice. Let $a\in I$ be an arbitrary set. In $W[[X]]$, consider the set $A$ of all well-orderings whose domain belongs to $I$. Note that for every set $b\in I$, every relation $R$ on $b$ in $V[[X]]$ belongs to $W[[X]]$ since $\pstab(b)\subset\stab(R)$; in particular, $b$ carries a well-ordering in $W[[X]]$. Let $E$ be the $\pstab(a)$-orbit equivalence relation on $A$, and let $B$ be the set $A/E$. Clearly, $\pstab(a)\subset\stab(A, E, B)$, so $A, E, B\in W[[X]]$ holds. In addition, for every $E$-class $C\subset A$, $\pstab(a)\subset\stab(C)$, so $B$ is even well-ordered in $W[[X]]$. By the axiom of well-ordered choice in $W[[X]]$, there is a selector $f$ on $B$. 

Let $b\in I$ be such that $\pstab(b)\subset\stab(f)$ and $a\subset b$. Use the dynamical closure assumption on the ideal to icrease $b$ if necessary so that $\forall x\in X\setminus b\ \exists\gamma\in\pstab(b)\ \gamma\cdot x\neq x$. I claim that $b$ is $a$-large. Indeed, suppose that $c\in I$ is a set. Let $\leq$ be a well-ordering on $c$, let $C\in B$ be the $\pstab(a)$-orbit of $\langle c, \leq\rangle$, and consider the value $f(C)$. Since $\pstab(b)\subset\stab(f, C)$, $\pstab(b)\subset\stab(f(C))$ holds. Now, $f(C)=\langle \gamma\cdot c, \gamma\cdot\leq\rangle$ for some $\gamma\in \pstab(a)$; since $\gamma\cdot\leq$ is a well-ordering on $\gamma\cdot c$, a transfinite induction argument along it shows that $\forall x\in \gamma\cdot c\ \pstab(b)\subset\stab(x)$. Since $b$ is dynamically closed, $\gamma\cdot c\subset b$. This concludes the proof that $b$ is $a$-large; (2) follows.
\end{proof}

\noindent There is a very natural concept intimately related to cofinal orbits.

\begin{definition}
A dynamical ideal $\Gamma\acts X, I$ \emph{supports a cone measure} if in the associated permutation model $W[[X]]$ there is a set $C\subset I$ which is cofinal with respect to inclusion and such that for every set $D\subset C$ in $W[[X]]$ there is $a\in I$ such that the set $\{b\in C\colon a\subseteq b\}$ is either a subset of $D$, or it is disjoint from $D$. The set $C$ is referred to as a \emph{support of the cone measure}.
\end{definition}

\noindent In other words the ideal supports a cone measure if in the associated permutation model, the cone filter on some cofinal subset of $I$ is an ultrafilter. While I do not have an attractive restatement of this property which does not refer to the permutation model, there is an implication towards the cofinal orbits.

\begin{theorem}
\label{conetheorem}
If a dynamical ideal $\Gamma\acts X, I$ supports a cone measure, then it has cofinal orbits.
\end{theorem}

\begin{proof}
Suppose that $C\subset I$ is a cofinal set in the permutation model supporting a cone measure. Suppose that $a\in I$ is any set; I must find an $a$-large set $b$. Enlarging $a$ if necessary, assume that $\stab(C)\subseteq\pstab(a)$. Towards a contradiction, assume that there is no $a$-large set $b\in I$. Working in $V[[X]]$, by transfinite recursion on ordinal $\ga$ build sets $b_\ga\in C$ such that, writing $B_\ga$ for the $\pstab(a)$-orbit of $b_\ga$, either $\bigcup_{\gb\in\ga}B_\gb\subset C$ is cofinal in $I$ (in which case the construction terminates), or $\bigcup_{\gb\in\ga}B_\gb$ does not contain any superset of $b_\ga$. In addition, require that $b_\ga\subseteq b_{\ga+1}$.

The construction clearly has to terminate at some point, since the sets $B_\ga\subset I$ are pairwise disjoint. Since there is no $a$-large set, no set $B_\ga$ is cofinal, therefore the termination stage must be a limit ordinal, call it $\gamma$. Consider the sets $C_0=\bigcup\{B_\ga\colon\ga\in\gg$ is even$\}$ and $C_1=\bigcup\{B_\ga\colon\ga\in\gg$ is odd$\}$.
The sets $C_0$ and $C_1$ are $\pstab(a)$-invariant subsets of $C$, therefore they belong to the permutation model. They are clearly disjoint. The requirement $b_\ga\subseteq b_{\ga+1}$ implies that the downward closure of the sets $C_0$ and $C_1$ is the same. Since $C_0\cup C_1\subseteq C$ is cofinal, it must be the case that the disjoint sets $C_0, C_1\subset C$ are both  cofinal, an impossibility if the cone filter on $C$ is to be an ultrafilter.
\end{proof}

\noindent Together with Theorem~\ref{cofinaltheorem}, Theorem~\ref{conetheorem} shows that in permutation models, existence of cone measure on some cofinal subset of $I$ implies the axiom of well-ordered choice. This is a curious implication, as its antecedent contradicts AC as soon as the ideal is not generated by a single set, while its decedent is a consequence of AC. I do not know if the reverse implication can fail for a dynamically complete ideal which does not contain the whole set $X$. All examples below support a cone measure, even though I identify its support only in the first one, where it relates to a famous theorem of geometric topology.

\begin{example}
\label{boundedexample}
Let $X$ be any Euclidean space, let $\Gamma$ be the Polish group of self-homeomorphisms of $X$, and let $I$ be the ideal of bounded subsets of $X$. Then $I$ has cofinal orbits. In dimension one, this is \cite[Model N47]{howard:ac}, in which apparently the status of well-ordered choice has not been known. The dynamical ideal supports a cone measure on the set $C=\{a\in I\colon\exists\gamma\in\Gamma\ \gamma\cdot a$ is the unit ball$\}$. To see that the cone filter on $C$ is an ultrafilter, suppose that $D\subset C$ is a set in the permutation model, and let $a\in I$ be a set such that $\stab(D)\subseteq\pstab(a)$. Enlarging the set $a$ if necessary, assume that it is equal to the closed ball $B_n$ of radius $n$ around the origin for some number $n$. Let $b_0, b_1\in C$ be some sets such that $B_{n+1}\subset b_0\cap b_1$. It will be enough to show that there is a homeomorphism $\gamma\in\pstab(a)$ such that $\gamma\cdot b_0=b_1$: the invariance of the set $D$ will then imply that $D$ contains either all or no sets in $C$ which contain $B_{n+1}$. Now, the annulus theorem \cite{kirby:annulus} shows that the sets $b_0\setminus B_n$ and $b_1\setminus B_n$ are both homeomorphic to the annulus $B_{n+1}\setminus B_n$. This allows a construction of homeomorphisms $\gamma_0, \gamma_1\in\pstab(a)$ such that $\gamma_0\cdot (B_{n+1}\setminus B_n)=b_0\setminus a$ and $\gamma_1\cdot B_{n+1}\setminus B_n)=b_1\setminus a$. The composition $\gamma_1\gamma_0^{-1}\in\pstab(a)$ then transports $b_0$ to $b_1$ as required.
\end{example}

\begin{example}
\cite[Model N33]{howard:ac} Let $X=\mathbb{Q}$ with its ordering, let $\Gamma$ be the Polish group of all automorphisms of $X$, and let $I$ be the ideal of bounded subsets of $X$. Then $I$ has cofinal orbits.
\end{example}

\begin{example}
\label{qcofinalexample}
Let $X=\mathbb{Q}$ with its ordering, let $\Gamma$ be the Polish group of all automorphisms of $X$, and let $I$ be the ideal of nowhere dense subsets of $X$. Then $I$ has cofinal orbits.
\end{example}

\begin{proof}
The argument consists of several simple back and forth constructions encapsulated in the following claims.

\begin{claim}
\label{easyclaim}
Let $d_0, d_1\subset\mathbb{Q}$ be two sets such that both they and their complements are dense in $\mathbb{Q}$. Then there is an order-preserving permutation $\pi$ of $\mathbb{Q}$ such that $\pi''d_0=d_1$.
\end{claim}

\begin{claim}
Let $a\subset\mathbb{Q}$ be a nowhere dense set. Then $a$ can be extended to a nowhere dense set which is in addition closed, coinitial and cofinal in $\mathbb{Q}$, and such that no point of it is isolated from left or right.
\end{claim}

\noindent Nowhere dense sets satisfying the properties spelled out in the last claim will be called \emph{good}. They form the support of the cone measure in the permutation model.

\begin{claim}
\label{triclaim}
Let $d_0, d_1\subset\mathbb{Q}$ be two nowhere dense good sets. Then there is an order-preserving permutation $\pi$ of $\mathbb{Q}$ such that $\pi''d_0=d_1$.
\end{claim}

\begin{proof}
Let $Y_0$ be the linearly ordered set whose universe consists of elements of $d_0$ and inclusion-maximal intervals of $\mathbb{Q}$ disjoint from $d_0$ and the ordering is the natural one inherited from $\mathbb{Q}$; the same for subscript $1$.
The goodness of the set $d_0$ implies that the linear ordering on $Y_0$ is dense without endpoints, $d_0\subset Y_0$ is dense, and so is $Y_0\setminus d_0\subset Y_0$. Claim~\ref{easyclaim} shows that there is an order-preserving bijection $\theta\colon Y_0\to Y_1$ such that $\theta''d_0=d_1$. For each interval $i\in Y_0\setminus d_0$, let $\pi_i\colon i\to \theta(i)$ be an order-preserving bijection. In the end, $\pi=(\theta\restriction d_0)\cup\bigcup\{\pi_i\colon i\in Y_0\setminus d_0\}$ is the desired permutation of $\mathbb{Q}$.
\end{proof}

Finally, let $a\subset X$ be a nowhere dense set; I must produce an $a$-large nowhere dense set $b$. Enlarging $a$ if necessary, I may assume that $a$ is good. Write $J$ for the set of all inclusion-maximal intervals of $X$ disjoint from $a$. Let $b\subset X$ be a good nowhere dense set such that $a\subset b$ and in each $i\in J$ the set $b$ is both coinitial and cofinal. I will show that $b$ is $a$-large.

To prove this, let $c\subset X$ be a nowhere dense set. Enlarging $c$ if necessary, I may assume that it contains $a$, it is good, and in each interval $i\in J$ it is coinitial and cofinal. For each interval $i\in J$ use Claim~\ref{triclaim} to find an order preserving permutation $\theta_i$ of $i$ such that $\theta_i''(b\cap i)=c\cap i$. Then $\pi=(\id\restriction a)\cup\bigcup\{\theta_i\colon i\in J\}$ is an order-preserving permutation of $X$ which is constant on $a$ and moves $b$ to $c$.
\end{proof}

\begin{example}
\label{cardinalitycofinalexample}
\cite[Model N12($\kappa$)]{howard:ac} Let $\kappa$ be an uncountable cardinal, let $X$ be a set of cardinality at least $\kappa$, let $\Gamma$ be the group of all permutations of $X$ acting by application, and let $I_\kappa$ be the ideal of subsets of $X$ of cardinality smaller than $\kappa$. Then $\Gamma\acts X, I_\kappa$ has cofinal orbits iff $\kappa$ is a successor cardinal.
\end{example}

\begin{proof}
If $\kappa$ is a successor cardinal, $\kappa=\lambda^+$, then for any set $a\in I_\kappa$ any set $b$ such that $a\subset b$ and $|b\setminus a|=\lambda$ has a cofinal orbit under the action of $\pstab(a)$. On the other hand, if $\kappa$ is a limit cardinal, then given a set $b\in I_\kappa$, its orbit consists of sets of cardinality $|b|$ only, but there are sets of cardinality greater than $|b|$ in the ideal $I_\kappa$. Thus, in that case $I_\kappa$ fails to have cofinal orbits.
\end{proof}

\begin{example}
Let $X$ be a set and let $I$ be an ideal on it. Say that $I$ is \emph{uniform} if there is an infinite cardinal $\kappa$ such that $I$ is generated by sets of cardinality $\kappa$, and for every set $a\in I$ there is a set $b\in I$ of cardinality $\kappa$ which is disjoint from $a$. If $I$ is a uniform ideal on $X$ and $\Gamma$ is the group of all permutations of $X$ with support in $I$ acting on $X$ by application, then $\Gamma\acts X, I$ has cofinal orbits.
\end{example}

\begin{proof}
Let $a\in I$ be an arbitrary set, and let $b\in I$ be a set disjoint from $I$ of the uniform cardinality $\kappa$. It will be enough to show that the $\supp(a)$-orbit of $a\cup b$ is cofinal. To this end, let $c\in I$ be an arbitrary set; enlarging $c$ if necessary, we may assume that $a\cup b\subset c$. Let $\pi\colon b\to c\setminus a$ be any bijection. Let $d=c\setminus b$. The uniformity assumption shows that there are pairwise disjoint sets $d_n$ for $n\in\gw$ such that $d_0=d$, all $d_n$ for $n>0$ are disjoint from $c$ and have cardinality $|d|$, and $\bigcup_nd_n\in I$.
For each $n\in\gw$ let $\pi_n\colon d_n\to d_{n+1}$ be a bijection. Then $\pi\cup\bigcup_n\pi_n$ is a permutation of a set in $I$ disjoint from $a$. Let $\gamma$ be the permutation of $X$ extending $\pi\cup\bigcup_n\pi_n$ by the identity outside of its domain. Then $\gamma\in\pstab(a)$ and $c\subset\gamma\cdot (a\cup b)$ as desired.
\end{proof}

\begin{example}
Let $n\geq 1$ be a number, let $X=\mathbb{R}^n$, let $I$ be the ideal of nowhere dense subsets of $X$, and let $\Gamma$ be the group of self-homeomorphisms of $X$ acting by application. Then the dynamical ideal $\Gamma\acts X, I$ has cofinal orbits. The $0$-large set is a variation of the Sierpinski carpet in $n$-dimensions. The Sierpinski carpets form a support for the cone measure in the permutation model  \cite{young:personal}.
\end{example}

\begin{example}
Let $X=\mathbb{R}^2$, let $I$ be the ideal of closed sets of dimension zero, and let $\Gamma$ be the group of all selfhomeomorphisms of $X$ acting by application. Then $\Gamma\acts X, I$ has cofinal orbits; the Cantor sets in the plane form the support of the cone measure  \cite{young:personal}.
\end{example}

\begin{example}
\label{countablenoncofinalexample}
Let $X$ be an completely metrizable space without isolated points, let $\Gamma$ be its self-homeomorphism group acting by application, and let $I$ be the ideal generated by countable closed subsets of $X$. The ideal does not have cofinal orbits: if $b\in I$ is a closed set, it has countable Cantor--Bendixson rank $\ga$, and its orbit consists of sets of rank $\ga$ only. However, there are sets in $I$ of higher Cantor--Bendixson rank. 

Now suppose that the space $X$ is the Euclidean space. The axiom of well-ordered (in fact, $\gw_1$) choice fails in the associated permutation model. Consider the sequence $\langle C_\ga\colon\ga\in\gw_1\rangle$ where $C_\ga$ is the set of closed countable subsets of $X$ of Cantor--Bendixson rank $\ga$. It is clear that the sequence is invariant under the group action and therefore belongs to the permutation model. I claim that the sequence has no selector in the permutation model. Suppose towards a contradiction that there is such a selector $s=\langle a_\ga\colon\ga\in\gw_1\rangle$, and let $b\subset X$ be a countable closed set such that $\pstab(b)\subset\stab(s)$. Select a countable ordinal $\ga$ greater than the Cantor--Bendixson rank of $b$, and a point $x\in a_\ga\setminus b$. The homogeneity properties of the Euclidean space imply that the $\pstab(b)$-orbit of $x$ contains a nonempty open set; in particular, it is uncountable. However, it should be a subset of the countable set $a_\ga$, a contradiction.
\end{example}

\section{The axiom of dependent choices}
\label{dcsection}

The axiom of dependent choices is one of the most commonly considered fragments of AC, partly because it allows an uneventful development of mathematical analysis and descriptive set theory. It is implied by well-ordered choice \cite{jech:cardinals} and it implies countable choice.

\begin{definition}
\cite[Form 43]{howard:ac} The \emph{axiom of dependent choices}, DC, is the statement that in every partial order there is either a minimal element or an infinite strictly descending sequence.
\end{definition}

There is an essentially equivalent restatement of DC in terms of dynamical properties of ideals. A version of this criterion for symmetric models has been considered by Karagila and Schilhan \cite{karagila:dc}.

\begin{definition}
\label{dccompletedefinition}
A dynamical ideal $\Gamma\acts X, I$ is \emph{DC-complete} if for every sequence $\langle a_n\colon n\in\gw\rangle$ of sets in the ideal $I$ there is a sequence $\langle \gg_n\colon n\in\gw\rangle$ of group elements such that $\gg_0\in\pstab(a_0)$, for every $n\in\gw$ and every $x\in a_n$ $\gg_n\cdot x=\gg_{n+1}\cdot x$, and $\bigcup_n\gamma_n\cdot a_n\in I$ holds.
\end{definition}

\begin{theorem}
\label{dctheorem}
Let $\Gamma\acts X, I$ be a dynamical ideal.

\begin{enumerate}
\item If the ideal is DC-complete, then the associated permutation model satisfies DC;
\item if the ideal is dynamically closed and the permutation model satisfies DC, then the ideal is DC-complete.
\end{enumerate}
\end{theorem}

\begin{proof}
For (1), assume that the ideal is DC-complete, and let $\langle P, \leq\rangle$ be a poset in the permutation model $W[[X]]$ with no minimal element. In the ambient model $V[[X]]$, pick an infinite strictly descending sequence $\langle p_n\colon n\in\gw\rangle$ of elements of $P$. Pick an inclusion-increasing sequence $\langle a_n\colon n\in\gw\rangle$ of sets in $I$ such that $\pstab(a_n)\subset\stab(P, \leq, p_n)$, and use the DC-completeness assumption to find a sequence $\langle\gamma_n\colon n\in\gw\rangle$ of group elements such that $\gamma_0\in\pstab(a_0)$, for every $n\in\gw$ $\gg^{-1}_{n+1}\gg_n\in\pstab(a_n)$, and $b=\bigcup_n\gamma_n\cdot a_n\in I$ holds. It will be enough to show that the sequence $\vec q=\langle \gg_n\cdot p_n\colon n\in\gw\rangle$ is an infinite strictly decreasing sequence in $P$ in the model $W[[X]]$.

To show that the sequence $\vec q$ belongs to the permutation model, it is enough to argue that $\pstab(b)\subseteq\stab(\vec q)$. To see that, note that for every number $n$ $\pstab(a_n)\subseteq\stab(p_n)$ holds by the choice of $a_n$, $\pstab(\gamma_n\cdot a_n)\subseteq\stab(\gamma_n\cdot p_n)$ by support invariance (Proposition~\ref{siproposition}(3)) and $\pstab(b)\subseteq\pstab(\gamma_n\cdot a_n)$ as $\gamma_n\subseteq b$. Thus elements of $\pstab(b)$ fix every entry of the sequence $\vec q$, so they fix the whole sequence $\vec q$.

To show that the sequence $\vec q$ is a strictly descending sequence in $P$, first note that for every number $n$, $\gg_n\in\pstab(a_0)$. (This is proved by induction on $n$, the base step being trivial, and the induction step relying on the identity $\gg_{n+1}=\gg_n (\gg_n^{-1}\gg_{n+1})$ and the fact that $\gg_n^{-1}\gg_{n+1}\in\pstab(a_n)\subset\pstab(a_0)$ holds.) It follows that the $n$-th entry $\gg_n\cdot p_n$ of the sequence $\vec q$ is in $P$, since $\gg_n\in\pstab(a_0)\subset\stab(P)$ and $p_n\in P$. Finally, $p_{n+1}<p_n$ implies $\gg_n^{-1}\gg_{n+1}\cdot p_{n+1}<\gg_n^{-1}\gg_{n+1}\cdot p_n=p_n$, which inequality provides $\gg_{n+1}\cdot p_{n+1}<\gg_n\cdot p_n$. This shows that $\vec q$ is in fact a  strictly decreasing sequence in $P$.

For (2), assume that the ideal is dynamically closed, the permutation model satisfies DC, let $\langle a_n\colon n\in\gw\rangle$ is a sequence of sets in the ideal $I$ and work to find a sequence of group elements with the properties of Definition~\ref{dccompletedefinition}. Enlarging the sets $a_n$ if necessary, assume that they are increasing with respect to inclusion. Pick a sequence $\langle \leq_n\colon n\in\gw\rangle$ such that $\leq_n$ is a well-ordering on $a_n$ and $\leq_{n+1}$ is an end-extension of $\leq _n$. Let $T$ be the tree of all finite sequences $t=\langle t(i)\colon i\in n\rangle$ such that $n$ is a number and there exists a sequence (a \emph{witness} for $t\in T$) $\langle \gg_i\colon i\in n\rangle$ such that $\gg_0\in\pstab(a_0)$, for each $i\in n-1$ $\gg_{i+1}^{-1}\gg_i\in\pstab(a_i)$, and $t(i)=\gg_i\cdot \langle a_i, \leq_i\rangle$.

\begin{claim}
$T\in W[[X]]$.
\end{claim}

\begin{proof}
For every number $i\in\gw$, the pair $\langle a_i, \leq_i\rangle$ belongs to $W[[X]]$, since all elements of $\pstab(a_i)$ fix it. This means that for every group element $\gamma\in\Gamma$, $\gamma\cdot\langle a_i, \leq _i\rangle$ is in the model $W[[X]]$. It in turn follows that elements of $T$ are finite sequences of elements of the model $W[[X]]$. Thus, $T\subset W[[X]]$ holds, and it will be enough to show that $\pstab(a_0)\subseteq\stab(T)$.

To this end, let $t=\langle t(i)\colon i\in n\rangle\in T$ as witnessed by a sequence $\langle \gg_i\colon i\in n\rangle$, and let $\gd\in\pstab(a_0)$ be any element. It will be enough to show that $\gd\cdot t(i)$ as witnessed by the sequence $\langle\gd\gg_i\colon i\in n\rangle$. To see this, first note that $\gd\cdot t(i)=\gd\gg_i\cdot \langle a_i, \leq_i\rangle$ as the group $\Gamma$ acts on $W[[X]]$ by $\in$-automorphisms. Also note that $\gd\gg_0\in\pstab(a_0)$ as $\pstab(a_0)\subset\Gamma$ is a subgroup, and $(\gd\gg_{i+1})^{-1})(\gd\gg_i)=\gg_{i+1}^{-1}\gg_i\in\pstab(a_i)$ by the choice of the group elements $\gg_i$. The claim follows.
\end{proof}

\begin{claim}
$T$ has no minimal element.
\end{claim}

\begin{proof}
If $t=\langle t(i)\colon i\in n\rangle\in T$ is a sequence with a witness $\langle \gg_i\colon i\in n\rangle$, then it can be extended by the pair $\gg_{n-1}\cdot \langle a_{n}, \leq_{n}\rangle$ and the resulting longer sequence will be in $T$ as witnessed by $\langle \gg_i\colon i\in n, \gg_{n-1}\rangle$.
\end{proof}

\noindent By DC in the permutation model $W[[X]]$, the tree $T$ has an infinite branch $b\in W[[X]]$, which is viewed as an infinite sequence such that for every $n\in\gw$, $b\restriction n\in T$.

\begin{claim}
There is an infinite sequence $c$ such that for each $n\in\gw$, $c\restriction n$ witnesses the statement $b\restriction n\in T$.
\end{claim}

\noindent This sequence $c$ need not be in the permutation model.

\begin{proof}
The sequence $c$ is built by recursion on $n$. Suppose that $c\restriction n$ has been built, and it is a witness to $b\restriction n\in T$. Let $\langle\gd_i\colon i\in n+1\rangle$ be a witness to $b\restriction n+1\in T$. It will be enough to show that $c^\smallfrown \gd_n$ is a witness to $b\restriction n+1\in T$. The only nontrivial point is to prove that $\gd_n c(n-1)^{-1}\in\pstab(a_{n-1})$ holds. To see this, note that $\gd_n c(n-1)^{-1}=(\gd_n\gd_{n-1}^{-1})(\gd_{n-1}c(n-1)^{-1})$ and here the former parenthesis is in $\pstab(a_n)$ as the $\gd$'s witness $b\restriction n+1\in T$, and the latter parenthesis is in $\pstab(a_n)$ since it fixes $\leq_n$ and $\leq_n$ is a well-ordering on $a_n$.
\end{proof}

Now, write $\gg_i=c(i)$ and consider the sets $\gg_i\cdot a_i$, and the well-orderings $\gg_i\cdot \leq_i$ for every number $i\in\gw$. For each $i\in\gw$, $\leq_{i+1}$ is a well-ordering which end-extends $\leq_i$. Acting on this statement with $\gg_i^{-1}\gg_{i+1}$ (which fixes $a_i$ pointwise and therefore fixes $\leq_i$) and then with $\gg_i$, it becomes clear that $\gg_{i+1}\cdot \leq_{i+1}$ is a well-ordering which end-extends $\gg_i\cdot\leq_i$. It follows that the set $\bigcup_i\gg_i\cdot a_i$ (which is in $W[[X]]$ as the union of the first coordinates of the sequence $b$) is well-ordered by the relation $\bigcup_i\gg_i\cdot\leq_i$ (which is in $W[[X]]$ as the union of the second coordinates of the sequence $b$). By the dynamical closure of the ideal $I$, it follows that the set $\bigcup_i\gg_i\cdot a_i$ belongs to the ideal $I$. Therefore, the sequence $c=\langle \gg_i\colon i\in\gw\rangle$ answers the demands of Definition~\ref{dccompletedefinition}, and the proof is complete.
\end{proof} 

\noindent Nonexamples of DC-completeness are easier to find than examples.

\begin{example}
If $I=\bigcup_nI_n$ where each collection $I_n$ is invariant, closed under subset, and not equal to $I$, then the ideal is not DC-complete no matter what the group action is. To see that, consider a sequence $\langle a_n\colon n\in\gw\rangle$ where for each $n\in\gw$, $a_n\in I\setminus I_n$ holds.
\end{example}

\noindent Such will be the case for the ideal of finite sets regardless of the action or the underlying infinite set.

\begin{example}
\label{dcclosedexample}
Let $X$ be a completely metrizable space without isolated points. Let $\Gamma$ be the group of all self-homeomorphisms of $X$ acting on $X$ by application, and let $I$ be the ideal of countable subsets of $X$ with countable closure. For $n\in\gw$ let $a_n\subset X$ be a nonempty countable closed set such that for each $n$, $a_n\subseteq a_{n+1}$ and no point of $a_n$ is isolated in $a_{n+1}$. The closure of $\bigcup_n a_n$ is a perfect subset of $X$, therefore uncountable. In addition, whenever $\gamma_n$ for $n\in\gw$ are self-homeomorphisms of the space $X$ such that $\forall n\ \forall x\in a_n\ \gamma_n\cdot x=\gamma_{n+1}\cdot x$, then it is still the case that no point of $\gamma_n\cdot a_n$ is isolated in $\gamma_{n+1}\cdot a_{n+1}$, so $\bigcup_n\gamma_n\cdot a_n$ still has uncountable closure and is not in the ideal. It follows that the dynamical ideal $\Gamma\acts X, I$ is not DC-complete. In the case of a Euclidean space $X$, DC fails in the permutation model $W[[X]]$ as it is impossible to find in that model a sequence of sets $\langle a_n\colon i\in I\rangle$ such that for each $n$, $a_n\subseteq a_{n+1}$ and no point of $a_n$ is isolated in $a_{n+1}$. Cf.\ Examples~\ref{countablenoncofinalexample} and ~\ref{countableclosedexample}.
\end{example}

\noindent My examples of DC-completeness are all implied by some stronger property.

\begin{example}
If $I$ is a $\gs$-ideal then the ideal is DC-complete no matter what the group action, as for every sequence $\langle a_n\colon n\in\gw\rangle$ of sets in $I$, the group elements $\gg_n=1$ will work as in Definition~\ref{dccompletedefinition}. In this case, it is not hard to show that the model $W[[X]]$ is closed under countable sequences in $V[[X]]$, so the DC in $W[[X]]$ follows from DC in $V[[X]]$.
\end{example}

\begin{theorem}
\label{cofdctheorem}
Let $\Gamma\acts X, I$ be a dynamical ideal. If the ideal has cofinal orbits then it is DC-complete.
\end{theorem}

\begin{proof}
The proof depends on two simple properties of cofinal orbits:

\begin{claim}
\label{oneclaim}
The set $\{\langle a, b\rangle\colon a, b\in I$ and $b$ is $a$-large$\}$ is $\Gamma$-invariant.
\end{claim}

\begin{proof}
Let $b$ be $a$-large and $\gd\in\Gamma$. To show that $\gd\cdot b$ is $\gd\cdot a$-large, let $c\in I$ be arbitrary. Since $b$ is $a$-large, there is $\gamma\in\pstab(a)$ such that $\gamma\cdot\gd^{-1}\cdot c\subset b$. Multiplying by $\gd$ on both sides, we get $\gd\gg\gd^{-1}\cdot c\subset\gd\cdot b$. Since $\gg\in\pstab(a)$, $\gd\gg\gd^{-1}\in\pstab(\gd\cdot a)$ holds and the claim follows.
\end{proof}

\begin{claim}
\label{twoclaim}
Let $b$ be $a$-large and $c\in I$. Then there is $\gamma\in\pstab(a)$ such that $\gamma\cdot c\subseteq b$ and $b$ is $\gamma\cdot c$-large.
\end{claim}

\begin{proof}
Let $d\in I$ be $c$-large; without loss $c\subset d$. Use the largeness assumption on $b$ to find a group element $\gamma\in \pstab(a)$ such that $\gamma\cdot d\subset b$. By Claim~\ref{oneclaim}, $\gamma\cdot d$ is $\gamma\cdot c$-large; since $b$ is a superset of $\gamma\cdot d$, it follows that $b$ is $\gamma\cdot c$-large as desired.
\end{proof}

\noindent To show the DC completeness, suppose that $\langle a_n\colon n\in\gw\rangle$ is a countable collection of elements of $I$ increasing with respect to inclusion. Let $b$ be $a$-large and aim to find group elements $\gamma_n$ such that $\forall x\in a_n\ \gamma_n(x)=\gamma_{n+1}(x)$, $\gamma_n\cdot a_n\subset b$ and $b$ is $\gamma_n\cdot a_n$-large. Let $\gamma_0=1$. If $\gamma_n$ has been found, consider the set $\gamma_n^{-1}\cdot b$, which by Claim~\ref{oneclaim} is $a_n$-large. Use Claim~\ref{twoclaim} to find $\gd\in\pstab(a_n)$ such that $\gd\cdot a_{n+1}\subseteq \gamma^{-1}_n$ and $\gamma^{-1}_n\cdot b$ is $\gd\cdot a_{n+1}$-large. The group element $\gamma_{n+1}=\gamma_n\gd$ works as required. In the end, $\bigcup_n\gamma_n\cdot a_n$ is a subset of $b$ and therefore belongs to the ideal $I$.
\end{proof}

\begin{example}
\cite{karagila:dc} Let $X$ be a countable dense linear ordering without endpoints, let $\Gamma$ be the group of all its automorphisms acting by application, and let $I$ be the ideal of nowhere dense sets. Then in the associated permutation model, DC holds--in fact, by Example~\ref{qcofinalexample}, the dynamical ideal has cofinal orbits, so the axiom of well-ordered choice holds in the permutation model.
\end{example}

\section{Stratifying well-ordered choice}
\label{countablesection}

There is an obvious way of stratifying the axiom of well-ordered choice into statements indexed by well-ordered infinite cardinals.

\begin{definition}
\cite[Form 8]{howard:ac}
The \emph{axiom of countable choice} is the statement that every countable family of nonempty sets has a choice function.  In general, if $\kappa$ is an infinite cardinal, then the \emph{axiom of $\kappa$-choice}, $\mathrm{AC}_\kappa$, is the statement that every family of cardinality $\kappa$ consisting of nonempty sets has a choice function.
\end{definition}

\noindent  This fragment of AC has a clean counterpart among properties of dynamical ideals:

\begin{definition}
Let $\Gamma\acts X, I$ be a dynamical ideal, and let $\kappa$ be an infinite cardinal. Say that the dynamical ideal is \emph{dynamically $\kappa$-complete} if for every set $a\in I$ and every collection $\langle b_\ga\colon \ga\in\kappa\rangle$ of sets in $I$ there are group elements $\gamma_\ga\in\pstab(a)$ for $\ga\in\kappa$ such that $\bigcup_{\ga\in\kappa}\gamma_\ga\cdot b_\ga\in I$.
\end{definition}

\begin{theorem}
Let $\Gamma\acts X, I$ be a dynamical ideal and $\kappa$ be an infinite cardinal.

\begin{enumerate}
\item If the dynamical ideal is dynamically $\kappa$-complete then the associated permutation model satisfies the axiom of $\kappa$-choice;
\item if the dynamical ideal is dynamically closed and the associated permutation model satisfies the axiom of $\kappa$-choice, then the ideal is dynamically $\kappa$-complete.
\end{enumerate}
\end{theorem}

\begin{proof}
For (1), assume that the ideal is dynamically $\kappa$-complete and $A=\langle A_\ga\colon \ga\in\kappa\rangle$ is a sequence of nonempty sets in the permutation model. Let $a\in I$ be such that $\pstab(a)\subset\stab(A)$. Outside of the permutation model, for each $\ga\in\kappa$ pick a set $B_\ga\in A_\ga$ and a set $b_\ga\in I$ such that $\pstab(b_\ga)\subseteq\stab(B_ga)$. Use the completeness assumption on the dynamical ideal to find elements $\gamma_\ga\in\pstab(a)$  for $\ga\in\kappa$ such that  $c=\bigcup_\ga\gamma_\ga\cdot b_\ga\in I$. Now, consider the sequence $C=\langle \gamma_\ga\cdot B_\ga\colon \ga\in\kappa\rangle$.
Since each $\gamma_\ga$ belongs to $\pstab(a)$, it is still the case that $\gamma_\ga\cdot B_\ga\in A_\ga$. At the same time, for each $\ga\in\kappa$ $\pstab(\gamma\cdot b_\ga)\subset\stab(\gamma_\ga\cdot B)$ holds by Proposition~\ref{siproposition}, so $\pstab(c)\subset\stab(C)$, $C$ belongs to the permutation model, and it witnesses the axiom of $\kappa$-choice instance for the collection $A$.

For (2), suppose that the ideal is dynamically closed and the permutation model satisfies $\kappa$-choice. To argue for the dynamical $\kappa$-completeness, suppose that that $a, b_\ga\colon \ga\in\kappa$ are sets in the ideal. Let $A_\ga$ be the set of all well-orderings on sets of the form $\gamma\cdot b_\ga$ where $\gamma$ ranges over all elements of $\pstab(a)$, and let $A=\langle A_\ga\colon\ga\in\kappa\rangle$.  Note that for every set of the form $\gamma\cdot b_\ga$ belongs to $I$, clearly $\pstab(\gamma\cdot b_\ga)$ fixes it pointwise, and therefore all relations on $\gamma\cdot b_\ga$ in $V[[X]]$, including well-orderings, belong to $W[[X]]$. Since $\pstab(a)\subset\stab(A_\ga)$ holds for every $\ga\in\kappa$, it must be the case that $A\in W[[X]]$. Use $\kappa$-choice to find a choice function $\langle \gamma_\ga\cdot b_\ga, \leq _\ga\colon \ga\in\kappa\rangle$ for $A$ in the premutation model. It is easy to see that the set $c=\bigcup_\ga\gamma_\ga\cdot b_\ga$ is well-ordered by an amalgamation of the well-orders $\leq_\ga$ in the permutation model. By the dynamical closure assumption on $I$ and Proposition~\ref{closureproposition}(2), it follows that the group elements $\gamma_\ga$ for $\ga\in\kappa$ witness the dynamical $\kappa$-completeness.
\end{proof}

\noindent The stratification effect stops at a clearly predetermined cardinal.

\begin{theorem}
Suppose that $\Gamma\acts X, I$ is a dynamical ideal. If $I$ is generated by $\kappa$ many sets, then in the permutation model $W[[X]]$,
AC$(\kappa)$ is equivalent to the well-ordered axiom of choice.
\end{theorem}

\begin{proof}
The right-to-left implication is clear. For the right-to-left implication, by Proposition~\ref{closureproposition}(1) without loss assume that $I$ is dynamically closed. Suppose $a\in I$ is a set. In $V[[X]]$, let $b_\ga\colon\ga\in\kappa$ be a collection generating the ideal $I$. For each $\ga\in\kappa$, let $A_\ga$ be the set of all pairs $\langle \gg\cdot b_\ga, \leq\rangle$ where $\gg\in\pstab(a)$ and $\leq$ is a well-ordering on $\gg\cdot b_\ga$. It is clear that the sequence $A=\langle A_\ga\colon\ga\in\kappa\rangle$ belongs to the permutation model, as $\pstab(a)\subseteq\stab(A)$. By AC$(\kappa)$ in the permutation model, there are elements $\gg_\ga\in\pstab(a)$ and well-orderings $\leq_\ga$ for $\ga\in\kappa$ such that the sequence $\langle \gg_\ga\cdot b_\ga, \leq_\ga\colon\ga\in\kappa\rangle$ is a selector for $A$ and belongs to the permutation model. The set $c=\bigcup_\ga\gg_\ga\cdot b_\ga$ is well-orderable in the premutation model by an amalgamation of the well-orders $\leq_\ga$. By the dynamical closure assumption and Proposition~\ref{closureproposition}(2), $c\in I$ holds. It is not hard to see that the set $c$ witnesses the instance of cofinal orbits for the set $a\in I$. By Theorem~\ref{cofinaltheorem}, the axiom of well-ordered choice holds in the permutation model.
\end{proof}

\noindent Thus, in the popular permutation models obtained from automorphisms of countable structures and the ideal of finite sets (Example~\ref{automorphismexample}), countable choice is in fact equivalent to well-ordered choice.

\begin{example}
\label{countableclosedexample}
Let $X$ be the closed unit interval $[0, 1]$, let $\Gamma$ be the group of orientation preserving self-homeomorphisms of $X$ acting by application, and let $I$ be the ideal of closed countable subsets of $X$. Then the ideal $\Gamma\acts X, I$ is dynamically $\aleph_0$-complete, but not dynamically $\aleph_1$-complete.
\end{example}

\noindent Compare this with Example~\ref{countablenoncofinalexample} and~\ref{dcclosedexample}: the associated permutation model satisfies the axiom of countable choice, but not the axiom of well-ordered choice or dependent choice. This example was improved by Justin Young \cite{young:personal} who showed that its conclusion holds for $X$ an arbitrary Euclidean space.

\begin{proof}
Let $a, b_n\colon n\in\gw$ be countable closed subsets of $X$. Let $J$ be the set of all inclusion-maximal open intervals disjoint from $a$. For each interval $j\in J$ and every $n\in\gw$, find an open interval $j(n)\subset j$ which is disjoint from $b_n$ and its endpoints are still in $j$, and find an oreintation-preserving self-homemorphism $h_{jn}\colon j\to j$ such that $j\setminus h_{jn}(j(n))$ consists of one interval at each end of $j$, both of length less than $2^{-n}$. For each number $n\in\gw$, let $\gamma_n=\id\restriction a\cup\bigcup_{j\in J}h_{jn}$; I will show that these group elements work as required.

It will be enough to show that $c=a\cup\bigcup_n\gamma_n\cdot b_n$ is a closed set, because it is clearly countable. Suppose then that $\langle x_k\colon k\in\gw\rangle$ is a converging sequence of points in $c$ with limit $x$, and work to show that $x\in c$. There are two cases.

\noindent\textbf{Case 1.} No interval $j\in J$ contains a tail of the sequence. In this case, the limit $x$ must be in the set $a$, completing the proof. 

\noindent\textbf{Case 2.} Case 1 fails. Let $j\in J$ be the unique interval containing a tail of the sequence. If the sequence has nonempty intersection with only fiinitely many sets $\gamma_n\cdot b_n$, then the limit must belong to one of those finitely many sets, completig the proof. If, on the other hand, the sequence has nonempty intersection with infinitely many sets  $\gamma_n\cdot b_n$, then by the choice of the homeomorphisms $h_{jn}$ it must be the case that the limit is equal to one of the endpoints of $j$. Both of these endpoints are in the set $a$. The example has been demonstrated.
\end{proof}

\begin{example}
\label{wocompleteexample}
\cite[Model N23]{howard:ac} Let $\langle X, \leq\rangle$ be a countable dense linear order without endpoints, let $\Gamma$ be its automorphism group, and let $I$ be the ideal of those subsets of $X$ which are well-ordered by $\leq$. The dynamical ideal $\Gamma\acts X, I$ is dynamically $\aleph_0$-complete.
\end{example}

\begin{proof}
Let $a\in I$ and $b_n$ for $n\in\gw$ be sets in the ideal; I must produce automorphisms $\gamma_n\in\pstab(a)$ for all $n\in\gw$ such that $a\cup \bigcup_n\gamma_n\cdot b_n\in I$ holds. Let $U$ be the set of all open intervals of $X$ which are disjoint from $a$ and maximal such. Inside each interval $i\in U$ pick points $x_{in}$ for $n\in\gw$ such that they form an increasing sequence cofinal in $i$, and automorphisms $\gamma_{in}$ equal to identity on $X\setminus i$ such that $\gamma_{in}(\min(b_n\cap i))\geq x_{in}$ whenever $b_n\cap i\neq 0$. Let $\gamma_n$ be the composition of all $\gamma_{in}$ for $i\in U$; this makes sense since the automorphisms $\gamma_{in}$ for $i\in U$ have pairwise disjoint supports. I claim that the set $c=a\cup \bigcup_n\gamma_n\cdot b_n$ is well-ordered as desired.

To see this, suppose towards a contradiction that $\langle x_m\colon m\in\gw\rangle$ is a strictly decreasing sequence in $c$. By the well-ordering assumption on $a$, the points $y_m=\min\{z\in a\colon z\geq x_m\}$ have to stabilize at some point; the tail of the sequence then belongs to the same open interval $i\in U$ which borders on its right end at the eventually stable value of the points $y_m$. By the cofinal choice of the points $x_{in}$ in $i$, there is an $n\in\gw$ such that a tail of the sequence consists of elements of $i$ which are smaller than $x_{in}$. However, this tail is then a subset of the set $\bigcup_{k\in n}\gamma_k\cdot b_k$. Since this is a finite union of well-ordered sets, it is itself well-ordered, and it cannot contain an infinite strictly decreasing sequence. A contradiction.
\end{proof}

\begin{example}
\cite[Model N10]{howard:ac} Let $\langle X, \leq\rangle$ be a countable dense linear order without endpoints, let $\Gamma$ be its automorphism group, and let $I$ be the ideal of those subsets of $X$ which are well-ordered by $\leq$, bounded in $X$, and bounded below every element of $X$. The dynamical ideal $\Gamma\acts X, I$ is dynamically $\aleph_0$-complete. The status of countable choice has apparently not been known in the associated permutation model.
\end{example}

\begin{proof}
Let $a\in I$ and $b_n$ for $n\in\gw$ be sets in the ideal; I must produce automorphisms $\gamma_n\in\pstab(a)$ for all $n\in\gw$ such that $a\cup \bigcup_n\gamma_n\cdot b_n\in I$ holds. Let $U$ be the set of all open intervals of $X$ which are disjoint from $a$ and maximal such. Note that the right endpoint of each interval $i\in u$ belongs to $a$, except for the rightmost, unbounded interval in $U$; in any case, the sets $b_n$ are all bounded in $i$. Inside each interval $i\in U$ pick points $x_{in}$ for $n\in\gw$ such that they form an increasing sequence which is bounded and does not have a supremum in $X$, pick points $y_{in}\in i$ such that $b_n\cap i$ is bounded by $y_{in}$, and pick automorphisms $\gamma_{in}$ equal to identity on $X\setminus i$ such that $\gamma_{in}(\min(b_n\cap i))\geq x_{in}$ whenever $b_n\cap i\neq 0$ and $\gamma_{in}(y_{in})=x_{in+1}$. Let $\gamma_n$ be the composition of all $\gamma_{in}$ for $i\in U$; this makes sense since the automorphisms $\gamma_{in}$ for $i\in U$ have pairwise disjoint supports. I claim that the set $c=a\cup \bigcup_n\gamma_n\cdot b_n$ is well-ordered and bounded below each point  as desired.

For the well-ordering, suppose towards a contradiction that $\langle x_m\colon m\in\gw\rangle$ is a strictly decreasing sequence in $c$. By the well-ordering assumption on $a$, the points $y_m=\min\{z\in a\colon z\geq x_m\}$ have to stabilize at some point; the tail of the sequence then belongs to the same open interval $i\in U$ which borders on its right end at the eventually stable value of the points $y_m$. By the choice of the points $x_{in}$ in $i$, there is an $n\in\gw$ such that a tail of the sequence consists of elements of $i$ which are smaller than $x_{in}$. However, this tail is then a subset of the set $\bigcup_{k\in n}\gamma_k\cdot b_k$. Since this is a finite union of well-ordered sets, it is itself well-ordered, and it cannot contain an infinite strictly decreasing sequence. A contradiction.

For the boundedness, let $z\in X$ be an arbitrary point. Since $a$ is bounded below $z$, $z$ is an internal point or the right-hand endpoint of some interval $i\in U$. If $z$ is greater than all points $x_{in}$ for $n\in\gw$, then as these points do not have a supremum in $X$, the sequence $\langle x_{in}\colon n\in\gw\rangle$ is bounded below $z$, and so is the set $c$. If there is a number $n$ such that $z\leq x_{in}$, then $c\cap i$ below $z$ is the union of the sets $\gamma_k\cdot b_k$ for $k\in n$, each of which is bounded below $z$, so $c$ is bounded below $z$ again. The proof is complete.
\end{proof}

\begin{example}
\cite[Model N21]{howard:ac}
Let $\kappa$ be an infinite regular cardinal. Let $X$ be the structure $(\kappa^+)^{<\gw}$ ordered by reverse extension. Let $\Gamma$ be the group of all automorphisms of $X$ acting by application, and let $I$ be the ideal generated by well-founded subtrees of $X$ of cardinality at most $\kappa$. Then $\Gamma\acts X, I$ is a dynamically $\kappa$-complete ideal. In the permutation model, $X$ witnesses the failure of DC, while $\kappa$-choice holds.
\end{example}

\begin{proof}
Let $a, b_\ga$ for $\ga\in\kappa$ be sets in the ideal $I$; increasing them if necessary, I may assume they are all well-founded trees. For every node $x\in a$ and every $\ga\in\kappa$, write $c(\ga, x)=\{y\in b_n\setminus a\colon y$ is an immediate successor of $x\}$. For each $\ga\in\kappa$ choose $\gamma_\ga\in\pstab(a)$ to be an automorphism of $X$ such that for every node $x\in a$, the sets $\gamma_\ga\cdot c(\ga, x)$ for $n\in\gw$ are pairwise disjoint. This is possible since the sets $c(\ga, x)$ have cardinality at most $\kappa$ while the set of immediate successors of $x$ has cardinality $\kappa^+$. I claim that the tree $a\cup\bigcup_\ga\gamma_\ga\cdot b_\ga\subset X$ is well-founded. Indeed, any putative infinite branch $v$ through it either must be a subset of $a$ (an impossibility by the well-foundedness assumption on $a$) or the set $v\cap a$ is finite and contains a longest element, call it $x$. If $y$ is the immediate successor of $x$ in $v$, then there is a unique number $\ga\in\kappa$ such that $y\in\gamma_\ga\cdot b_\ga$ by the choice of the automorphisms $\gamma_\ga$. This means that the branch $v$ must be a subset of $\gamma_\ga\cdot b_\ga$, contradicting the assumed well-foundedness of $b_\ga$.
\end{proof}

\section{Simplicity}
\label{simplesection}

\noindent One fragment of axiom of choice for which an attractive dynamical criterion can be found is well-orderable closure.

\begin{definition}
\cite[Form 231]{howard:ac}
\emph{Well-orderable closure} is the statement that for every well-orderable set $A$ consisting of well-orderable sets, $\bigcup A$ is well-orderable.
\end{definition}

\noindent This is often verified in permutation models as a consequence of  well-ordered axiom of choice (Section~\ref{wosection}). In situations where this is not available, the following elegant criterion on dynamical ideals is helpful.

\begin{definition}
A dynamical ideal $\Gamma\acts X, I$ is 

\begin{enumerate}
\item \emph{simple} if for every $a\subseteq b\in I$, the only normal subgroup of $\pstab(a)$ containing $\pstab(b)$ is $\pstab(a)$ itself;
\item \emph{almost simple} if for every set in $I$ has a superset $a\in I$ such that for every $a\subseteq b\in I$, the only normal subgroup of $\pstab(a)$ containing $\pstab(b)$ is $\pstab(a)$ itself.
\end{enumerate}
\end{definition}

\noindent Clearly, every simple dynamical ideal is almost simple. While the weaker version is sufficient to imply well-orderable closure, the full simplicity has the advantage of being natural and closed under the subideal operation. One natural example of an almost simple ideal which is not simple is introduced in Example~\ref{weglorzexample}.

\begin{theorem}
\label{simpletheorem}
Suppose that $\Gamma\acts X, I$ is an almost simple dynamical ideal. Then in the associated permutation model, a union of well-orderable set of well-orderable sets is well-orderable.
\end{theorem}

\begin{proof}
Suppose that $\Gamma\acts X, I$ is simple. Let $A$ be a well-orderable set consisting of well-orderable sets in the permutation model. By Proposition~\ref{wotheorem}, there is a set $a\in I$ such that $\pstab(a)\subset\pstab(A)$. Enlarging $a$ if necessary, I may assume that for every set $b\in I$ which is a superset of $a$, the only normal subgroup of $\pstab(a)$ containing $\pstab(b)$ is $\pstab(a)$ itself.  It will be enough to show that $\pstab(a)\subset\pstab(\bigcup A)$. In other words, given a set $B\in A$ and $C\in B$, I must prove that $\pstab(a)\subset\stab(C)$. Let $\gd\in\pstab(a)$ be any element and work to show that $\gd\cdot C=C$.

To prove this, use Proposition~\ref{wotheorem} again to find a set $b\in I$ such that $a\subseteq b$ and $\pstab(b)\subset\pstab(B)$. By the choice of the set $a$, there is a number $n\in\gw$ and elements $\gamma_m\in\pstab(a)$ and $\gd_m\in\pstab(b)$ for $m\in n$ such that $\gd=\prod_{m\in n}\gamma_m^{-1}\gd_m\gamma_m$. It will be enough to show that for every $m\in n$, $\gamma_m^{-1}\gd_m\gamma_m\cdot C=C$. Towards this, observe that $\gd_m$ fixes all elements of $B$' in particular, it fixes $\gamma_m\cdot C$. The proof is complete.
\end{proof}

\noindent Simplicity is an attractive concept with an attractive conclusion. However, as in the case of simplicity of groups, it is not always easy to check. The first class of examples comes from model theory.

\begin{definition}
Let $\mathcal{F}$ be a Fraiss{\' e} class with disjoint amalgamation. Say that $\mathcal{F}$ has \emph{hereditary canonical amalgamation} if there is a function $C$ which to each ordered pair of $\mathcal{F}$-structures $\langle A, B\rangle$ in amalgamation position assigns a minimal disjoint amalgamation so that

\begin{enumerate}
\item $C$ is invariant under isomorphism in both variables: if $\phi\colon A\to A'$ and $\psi\colon B\to B'$ are isomorphisms which agree on $\dom(A)\cap \dom(B)$, then there is an isomorphism of $C(A, B)$ to $C(A', B')$ extending $\phi\cup\psi$;
\item $C$ is hereditary for substructures in the left variable: if $A'$ is an induced substructure of $A$ such that $\dom(A)\cap\dom(B)\subset\dom(A')$, then there is an isomorphism between $C(A', B)$ and the algebraic closure of $\dom(A')\cup\dom(B)$ in $C(A, B)$ which is the identity on $\dom(A')\cup\dom(B)$.
\end{enumerate}
\end{definition}

\noindent Compared to the canonical amalgamation of Paolini and Shelah \cite{paolini:index}, the hereditary clause is added.

\begin{theorem}
\label{fraissetheorem}
Let $\mathcal{F}$ be a Fraiss{\' e} class with hereditary canonical disjoint amalgamation, let $X$ be its limit structure, let $\Gamma$ be the group of all automorphisms of $X$ acting by application, and let $I$ be the ideal of finite sets on $X$. Then $\Gamma\acts X, I$ is a simple dynamical ideal.
\end{theorem}

\begin{proof}
As a matter of terminology, say that a finite set $a\subset X$ is \emph{algebraically closed} if it is closed under all functions of $X$; $\mathrm{acl}(a)$ is the smallest algebraically closed subset of $X$ containing $a$. Let $C$ be the canonical amalgamation function on $\mathcal{F}$. Consider $\Gamma$ as a Polish group with the usual automorphism group topology. Let $a\subseteq b$ be finite subsets of $X$; I must prove that within $\pstab(a)$, the normal subgroup $\Delta$ generated by $\pstab(b)$ is equal to $\pstab(a)$ itself. I may assume that the sets $a, b$ are both algebraically closed. Since $\Delta$ is a group generated by an open subset of $\pstab(a)$, it is open, therefore closed. Thus, it is enough to show that $\Delta$ is dense in $\pstab(a)$.

To this end, let $\pi\colon c\to d$ be any morphism between two finite subsets of $X$, both including $a$ as a subset, both algebraically closed, and such that $\pi\restriction a=\id$. I must find an element of $\Delta$ extending $\pi$. To do this, let $e=\mathrm{acl}(b\cup c\cup d)$. By the saturation properties of $X$, there must be a finite set $f\subset X$ such that 

\begin{itemize}
\item $f\cap e=a$;
\item there is an isomorphism $\theta\colon e\to f$ which is identity on $a$;
\item $X\restriction \mathrm{acl}(e\cup f)$ is isomorphic to $C(X\restriction e, X\restriction f)$.
\end{itemize}

\noindent Let $\gd\in\pstab(a)$ be an automorphism extending $\theta$. By the heredity and invariance properties of $C$, there is a morphism between the algebraic closures of $c\cup f$ and $d\cup f$ extending the function $\pi\cup\id_f$. Let $\gamma\in\pstab(f)$ be any automorphism of $X$ extending it. Now, $\gamma\in\pstab(f)$ implies that $\gd^{-1}\gamma\gd\in\pstab(e)\subseteq\pstab(b)$. It follows that $\gamma=\gd(\gd^{-1}\gamma\gd)\gd^{-1}$ is an element of $\Delta$ extending $\pi$ as desired.
\end{proof}

\noindent Nearly all permutation models associated with limits of Fraiss{\' e} structures allow of a much more detailed analysis than just simplicity; this will appear in forthcoming work. For now, I will list only two very special examples.

\begin{example}
Let $\mathcal{F}$ be the class of structures with no relations and no functions; hereditary canonical amalgamation obviously holds.  Let $X$ be its limit, the countable set with no structure. The associated model is \cite[Model N1]{howard:ac}. In this model, $[X]^{<\aleph_0}$ is a set of well-orderable sets without a selector, showing that the conclusion of Theorem~\ref{simpletheorem} cannot be strengthened.
\end{example}

\begin{example}
The class of rational ultrametric spaces has hereditary canonical amalgamation. Given any two finite rational ultrametric spaces $A, B$ in the amalgamation position with nonempty intersection, with their metrics denoted by $d_A$ and $d_B$ respectively, define the metric $d$ on $\dom(A)\cup\dom(B)$ by letting $d_A\cup d_B\subset d$, and for points $x\in \dom(A)\setminus\dom(B)$ and $y\in\dom(B)\setminus\dom(A)$, the distance $d(x,y)$ is defined in two cases. If there is a point $z\in\dom(A)\cap\dom(B)$ such that $d_A(x, z)\neq d_B(z, y)$, then set $d(x, y)=\{d_A(x, z), d_B(z, y)\}$, and if there is no such $z$, then let $d(x, y)=\min\{d_A(z, y)\colon z\in\dom(A)\cap\dom(B)\}$. It must be verified that the definition is sound and yields an ultrametric space. It is clear that such amalgamation is hereditary and canonical.
\end{example}

\begin{example}
The class of vector spaces over a fixed finite field has hereditary canonical amalgamation: $C(A, B)$  will be the vector space with the underlying set $\dom(A)\times\dom(B)$ modulo the equivalence relation $E$ defined by $\langle x, y\rangle\mathrel{E}\langle x', y,\rangle$ if $x-x'$ and $y'-y$ are identical elements of $\dom(A)\cap\dom(B)$. Addition and scalar multiplication is defined coordinatewise.
\end{example}

\begin{example}
Let $\mathcal{F}$ be the set of finite structures with a function (coded as a relation) which from each unordered quadruple selects an unordered pair. Then $\mathcal{F}$ does not have canonical amalgamation: there is no isomorphism-invariant way to amalgamate disjoint sets $A$ and $B$ such that $A$ has three elements and $B$ has one. 
\end{example}

\noindent Finally, the collection of relational Fraisse classes with hereditary canonical disjoint amalgamation is closed under superposition, which yields interesting classes such as linearly ordered rational ultrametric spaces and the like.

Let me now turn to more set-theoretic examples of simple dynamical ideals.

\begin{example}
\label{ordersimpleexample}
Let $X$ be an ultrahomogeneous linear order , let $\Gamma$ be the group of all automorphisms of $X$ acting by application, and let $I$ be the ideal of nowhere dense sets. Then $\Gamma\acts X, I$ is a simple dynamical ideal. 
\end{example}

\begin{proof}
The argument uses the following key claim:

\begin{claim}
Let $a\subset X$ be a nowhere dense set, and let $\gamma\in\Gamma$ be an automorphism whose orbits are both cofinal and coinitial in $X$. Then there are $\gamma_i\in\pstab(a)$ and $\gd_i\in\Gamma$ for $i\in 4$ such that $\gamma=\prod_{i\in 4}\gd_i^{-1}\gamma_i\gd_i$.
\end{claim}

\begin{proof}
Without loss of generality assume that for all $x\in X$, $\gamma\cdot x>x$; the proof for $\gamma\cdot x<x$ is symmetric. Let $\langle x_n\colon n\in\mathbb{Z}\rangle$ be an increasing enumeration of one of the orbits. The ultrahomogeneity assumption shows that there are automorphisms $\ga_0, \ga_1\in\Gamma$ such that $a_0=\ga_0\cdot a\subset\bigcup_n (x_{4n-1}, x_{4n+1})$ and $a_1=\ga_1\cdot a\subset\bigcup_n(x_{4n+1}, x_{4n+3})$. I will express $\gamma$ as a composition of four automorphisms in $\pstab(a_0)\cup\pstab(a_1)$. Since $\pstab(a_0)=\ga_0\pstab(a)\ga_0^{-1}$ and
$\pstab(a_1)=\ga_1\pstab(a)\ga_1^{-1}$, this will conclude the proof.

First, use the ultrahomogeneity assumption on $X$ to argue that there is an automorphism $\gd_0\in\pstab(a_0)$ fixing $x_{4n}$ and such that $\gd_0(x_{4n+1})=x_{4n+2}$ and $\gd_0(x_{4n+2})=x_{4n+3}$ for all $n\in\mathbb{Z}$. Similarly, let $\gd_1\in\pstab(a_1)$ be an automorphism fixing $x_{4n+2}$ and $x_{4n+3}$ and such that $\gd_1(\gd_0(x_{4n+3}))=x_{4n+4}$ and $\gd_1(x_{4n})=x_{4n+1}$ for all $n\in\gw$. Let $\gb=\gd_1\gd_0$ and observe that $\{x_n\colon n\in\mathbb{Z}\}$ is an orbit of $\gb$. Now, let $\gd_2\in\pstab(a_0)$ be an automorphism which fixes the interval $(x_{4n-1}, x_{4n+1})$ pointwise and on $(x_{4n+1}, x_{4n+3})$ it is equal to $\gg\gb^{-1}$, this for all $n\in\mathbb{Z}$. Similarly, let $\gd_3\in\pstab(a_1)$ be an automorphism which fixes the interval $(x_{4n+1}, x_{4n+3})$ pointwise and on $(x_{4n-1}, x_{4n+1})$ it is equal to $\gg\gb^{-1}$, this for all $n\in\mathbb{Z}$. Note that $\gd_2$ and $\gd_3$ have disjoint supports, therefore they commute and $\gd_2\gd_3=\gd_3\gd_2=\gg\gb^{-1}$. In consequence, $\gg=\gd_3\gd_2\gd_1\gd_0$ as desired.
\end{proof} 

\noindent Now, suppose that $a\subset b$ are sets in the ideal $I$, let $\gamma\in\pstab(a)$, and work to show that $\gamma$ belongs to the normal subgroup of $\pstab(a)$ generated by $\pstab(b)$. Let $U$ be the set of all intervals of $X$ in which $\gamma$ has an orbit which is both coinitial and cofinal in $U$. Note that such intervals do not have smallest or largest elements, therefore they are ultrahomogeneous themselves, and every orbit is both coinitial and cofinal in them. Note also that $a\subset X\setminus\bigcup U$ and $\gamma$ fixes all elements of $X\setminus\bigcup U$. Now, note that each interval $u\in U$ is order-isomorphic to $X$; write $\Gamma_u$ for its group of automorphisms. Apply the claim with $u$ and $\Gamma_u$ instead of $X$ and $\Gamma$, and with the nowhere dense set $a\cap u\subset u$ to find automorphisms $\gamma_{iu}\in\pstab(a\cap u)$ in $\Gamma_u$ and $\gd_{iu}\in\Gamma_u$ for $i\in 4$ such that $\gamma\restriction u=\prod_{i\in 4}\gd_{iu}^{-1}\gamma_{iu}\gd_{iu}$. Finally, for each $i\in 4$ let $\gamma_i\in\Gamma$ be the union of all $\gamma_{iu}$ for $u\in U$ together with the identity on $X\setminus\bigcup U$, and similarly for $\gd_i$. It is clear that $\gamma=\prod_{i\in 4}\gd_i^{-1}\gamma_i\gd_i$, so $\gamma\in\Delta$ as desired.
\end{proof}

\begin{example}
\label{cardinalitysimpleexample}
\cite[Model N16]{howard:ac} Let $\kappa$ be an uncountable cardinal, let $X$ be a set of cardinality at least $\kappa$, let $\Gamma$ be the group of all permutations of $X$ acting by application, and let $I_\kappa$ be the ideal of subsets of $X$ of cardinality smaller than $\kappa$. Then $\Gamma\acts X, I_\kappa$ is a simple dynamical ideal. The status of well-orderable closure for $\kappa$ of countable cofinality has apparently not been known.
\end{example}

\begin{proof}
Let $a\subseteq b$ be sets in $I_\kappa$, let $\Delta$ be the normal subgroup of $\pstab(a)$ generated by $\pstab(b)$, let $\gamma\in\pstab(a)$ be an arbitrary group element, and work to show that $\gamma\in\Delta$. Let $c\in I$ be a set containing $b$ and such that $c$ is closed under $\gamma$. Let $\gb\in\pstab(a)$ be a group element such that $\gb\cdot b\cap c=a$.
Then $(\gamma\restriction c)\cup (\id\restriction\gd\cdot b)$ is a function defined on a set of cardinality smaller than $\kappa$; let $\ga\in\pstab(a)$ be any permutation of $X$ extending it.

Now, $\ga\in\pstab(\gb\cdot b)$ holds, so $\gd^{-1}\ga\gd\in\pstab(b)$ holds, and $\ga=\gd(\gd^{-1}\ga\gd)\gd^{-1}$ is an element of $\Delta$. At the same time, $\ga=\gg$ on the set $c$, so $\ga^{-1}\gg\in\pstab(c)\subset\pstab(b)$ must hold. It follows that $\gg=\ga(\ga^{-1}\gg)$ belongs to $\Delta$ as required.
\end{proof}

\begin{example}
\cite[Model N21]{howard:ac}
Let $\kappa$ be an uncountable regular cardinal. Let $X=\kappa^{<\gw}$ partially ordered by reverse extension, let $\Gamma$ be the group of all automorphisms of $X$ acting by application, and let $I$ be the ideal generated by well-founded subtrees of $X$ of cardinality smaller than $\kappa$. Then $\Gamma\acts X, I$ is a simple dynamical ideal. The status of well-orderable closure in the associated permutation model apparently has not been known.
\end{example}

\begin{proof}
Let $a\subseteq b$ be sets in the ideal $I$. Since the dynamical closure of any subset of $X$ is a subtree of $X$, I may assume that both $a$ and $b$ are in fact well-founded trees. Let $\Delta$ be the normal subgroup of $\pstab(a)$ generated by $\pstab(b)$ and let $\gamma\in\pstab(a)$ be an arbitrary automorphism of $X$; I must show that $\gamma\in\Delta$ holds.

Let $c\subset X$ be any tree of cardinality smaller than $\kappa$ which contains $b$ and which is closed under $\gamma$ and its inverse. Let $\gd\in\pstab(a)$ be an automorphism such that $c\cap\gd\cdot b=a$. Then the map $(\gamma\restriction c)\cup(\id\restriction \gd\cdot b)$ is an automorphism of a subtree of $X$ of cardinality smaller than $\kappa$. As such, it can be extended to a full automorphism $\ga\in\pstab(a)$. Then, $\ga\in\pstab(\gd\cdot b)$, so $\gd^{-1}\ga\gd\in\pstab(b)$, and $\ga=\gd(\gd^{-1}\ga\gd)\gd^{-1}\in\Delta$. At the same time, $\ga=\gg$ on the set $c$, so $\ga^{-1}\gg\in\pstab(c)\subset\pstab(b)$ must hold. It follows that $\gg=\ga(\ga^{-1}\gg)$ belongs to $\Delta$ as required.
\end{proof}

\noindent Last, but not least, the concept of simple dynamical ideal has been formulated in such a way that it survives many common operations on ideals. I will state only the most obvious:

\begin{example}
A subideal of a simple dynamical ideal is simple.
\end{example}

\noindent In view of Example~\ref{ordersimpleexample}, this includes such dynamical ideals as the following:

\begin{example}
\cite[Model N23]{howard:ac} Let $X$ be an ultrahomogeneous linear order, let $\Gamma$ be the group of all automorphisms of $X$ acting by application, and let $I$ be the ideal of sets which are well-ordered by the linear order of $X$. Then $\Gamma\acts X, I$ is a simple dynamical ideal. The status of well-orderable closure in this model has apparently not been known.
\end{example}

\begin{example}
\label{unionexample}
An increasing union of simple dynamical ideals is again simple.
\end{example}

\noindent Finally, a couple of natural examples of dynamical ideals which are not simple.

\begin{example}
Let $X=\mathbb{R}^2$, let $\Gamma$ be the homeomorphism group of $X$ acting by application, and let $I$ be the ideal of bounded sets. Then $I$ has cofinal orbits by Example~\ref{boundedexample}. However, it is not simple: let $a$ be the set containing only the origin, and let $b_\eps$ be the closed disc of radius $\eps$. It is not difficult to check that the normal subgroup of $\pstab(a)$ generated by $\pstab(b_1)$ is equal to $\bigcup_{\eps>0}\pstab(b_\eps)$. Any nontrivial rotation of $\mathbb{R}^2$ shows that this subgroup is not equal to $\pstab(a)$.
\end{example}

\begin{example}
\label{weglorzexample}
\cite[Model N51]{howard:ac} Let $X$ be the set $\power(\gw)$ equipped with the subset relation. Let $\Gamma$ be the group of all automorphisms of $X$, and let $I$ be the ideal of finite subsets of $X$. Then $\Gamma\acts X, I$ is an almost simple, not simple dynamical ideal. The status of well-orderable closure in the associated permutation model was apparently unkown.
\end{example}

\begin{proof}
It is immediate that every element of $\Gamma$ is determined by its action on singletons, so $\Gamma$ is naturally isomorphic to $S_\infty$. The Boolean operations of union, intersection, and complement on $X$ are invariant under the action.

To show that the dynamical ideal is not simple, let $x\in X$ be any set containing exactly two elements. Let $y_0, y_1$ be the singletons formed by these two elements and observe that $\pstab(\{y_0, y_1\})$ is a proper normal subgroup of $\pstab(\{x\})$. To show that the dynamical ideal is almost simple, let $a\subset X$ be a finite Boolean subalgebra such that for every finite set $x\in a$ also contains all singletons formed by elements of $a$. It will be enough to show that for every finite Boolean algebra $b\subset X$ containing $a$, the normal subgroup $\Delta\subseteq\pstab(a)$ generated by $\pstab(b)$ is equal to $\pstab(a)$.

To see this, for every infinite atom $x\in a$ consider the group $\Delta_x$ of all elements of $\Gamma$ fixing each element $\{n\}\in X$ for every $n\in\gw\setminus x$. Clearly, $\Delta_x$ is naturally isomorphic to the permutation group on $x$, therefore to $S_\infty$; it is also a subset of $\pstab(a)$. The group $\pstab(b)\cap\Delta_x$ is an uncountable subgroup of $\Delta_x$. Since the only uncountable normal subgroup of $S_\infty$ is $S_\infty$ itself, it must be the case that $\Delta_x\subset\Delta$ holds. Finally, the group $\pstab(a)$ is clearly generated by the union of the groups $\Delta_x$ as $x$ varies over all infinite atoms of $a$.  It follows that $\Delta=\pstab(a)$ as desired.
\end{proof}

\noindent On the opposite end of the spectrum lurk the dynamical ideals associated with abelian groups. This is the content of the following theorem.

\begin{theorem}
\label{abeliantheorem}
Let $\Gamma\acts X, I$ be an abelian dynamical ideal. Then, in the associated permutation model,

\begin{enumerate}
\item every set is a union of a well-orderable collection of well-orderable sets;
\item the axiom of choice for families of finite sets implies the full axiom of choice.
\end{enumerate}
\end{theorem}

\begin{proof}
For (1), let $A\in W[[X]]$ be a set, and let $a\in I$ be a set such that $\pstab(a)\subset\stab(A)$. I need to produce a well-orderable set $B$ such that $B$ is well-orderable, every element of $B$ is, and $\bigcup B=A$. To this end, let $B$ be the set of all $\pstab(a)$-orbits of elements of $A$. Every element of $B$ is $\pstab(a)$-invariant, so $B$ belongs to the permutation model and it is well-orderable there by Proposition~\ref{wotheorem}. Now, let $C\in B$ be an arbitrary set, let $D\in C$ be arbitrary, and let $d\in I$ be such that $a\subset d$ and $\pstab(d)\subset\stab(D)$. It will be enough to show that $\pstab(d)$ fixes all elements of $C$--then, $C$ is well-orderable in the permutation model by Proposition~\ref{wotheorem} again.

Thus, suppose that $\gamma\in\pstab(d)$ is arbitrary. Let $\gd\in\pstab(a)$ be arbitrary; I must show that $\gamma\gd\cdot D=\gd\cdot D$. To this end, use the commutativity of the group $\Gamma$ to conclude that the left-hand side of the equation is equal to $\gd\gamma\cdot D$ which is equal to $\gd\cdot D$ as $\gamma\in\stab(D)$. (1) follows.

For (2), in $W[[X]]$, consider the set $A$ which contains every nonempty finite subset of $X$ and every unordered pair of nonempty disjoint subsets of $X$. It will be enough to show that whenever $f$ is a selector function on $A$ and $b\in I$ is a set such that $\pstab(b)\subset\stab(f)$, then $\pstab(f)=\pstab(X)$. To do this, fix $f$ and $b$ as assumed, and towards a contradiction assume that there are elements $\gamma\in\pstab(b)$ and $x\in X$ such that $\gamma\cdot x\neq x$. Write $c$ for the orbit of $x$ under $\gamma$ and $\gamma^{-1}$. There are two cases.

\noindent\textbf{Case 1.} The set $c\subset X$ is finite. Then $c\in\dom(f)$, both $f$ and $c$ are fixed by $\gamma$, but no element of $c$ is. This must include the element $f(c)\in c$, contradicting the fact that the action by $\gamma$ is an $\in$-automorphism of the permutation model.

\noindent\textbf{Case 2.} The set $c\subset X$ is infinite. In such a case, define an equivalence relation $E$ on $X$ connecting elements $y, z\in X$ if there is an even integer $n$ such that $y=\gamma^n\cdot z$. The equivalence relation $E$ is invariant under the group action: whenever $\gd\in\Gamma$ is arbitrary and $y=\gamma^n\cdot z$ then $\gd\cdot y=\gd\gamma^n\cdot z=\gamma^n\gd\cdot z$, where the last equality follows from commutativity of the group $\Gamma$. It follows that $E$ belongs to the permutation model. Consider the set $d$ of the two $E$-equivalence classes represented in the set $c$. It is clear that $d\in\dom(f)$, $\gamma$ fixes both $f$ and $d$. It is also clear that $\gamma$ flips the two elements of $d$, so it moves $f(d)$. This again contradicts the fact that the action by $\gamma$ is an $\in$-automorphism of the permutation model.
\end{proof}

\begin{example}
\cite[Model N2(LO)]{howard:ac} Let $X=\gw\times \mathbb{Z}$, let $\Gamma=\mathbb{Z}^\gw$ be the abelian group acting on $X$ by $\gamma\cdot \langle n, z\rangle=\langle n, \gamma(n)+z\rangle$, let $I$ be the ideal generated by the vertical sections of $X$. Axiom of choice fails in the associated permutation model; thus, there must be a collection of finite sets for which there is no choice function. To find it, for each $n\in\gw$ let $E_n$ be the equivalence relation on $n$-th vertical section of $X$ defined by $\langle n, z_0\rangle E_n\langle n, z_1\rangle$ if $z_0-z_1$ is an even integer. Let $Y_n$ be the set of the two $E_n$-equivalence classes. It is not difficult to see that in the permutation model, $\langle Y_n\colon n\in\gw\rangle$ is a sequence of two-element sets without a choice function.
\end{example}

\subsection{Dedekind finite and amorphous sets}
\label{finitesection}

In the broad field of definitions of finiteness, the following definitions stand out:

\begin{definition}
A set is \emph{Dedekind finite} if it does not contain an injective image of $\gw$. A set is \emph{amorphous} if it cannot be partitioned into two infinite sets.
\end{definition}

\noindent Non-existence of amorphous or infinite, Dedekind finite sets is one of the more common fragments of axiom of choice \cite[Form 9]{howard:ac}. The purpose of this section is to provide a dynamical criterion related to the previous concepts which implies it.

\begin{definition}
\label{stratifieddefinition}
A dynamical ideal $\Gamma\acts X, I$ is said to be \emph{stratified} if there are $\Gamma$-invariant ideals $I_n$ on $X$ for $n\in\gw$ such that

\begin{enumerate}
\item $I$ is the increasing union of all $I_n$'s;
\item for every $n\in\gw$ and all sets $a, b_m\colon m\in\gw$ in $I_n$ there are elements $\gamma_m\in\pstab(a)$ such that $\bigcup_m\gamma_m\cdot b_m\in I_{n+1}$.
\end{enumerate}
\end{definition}

\begin{theorem}
Let $\Gamma\acts X, I$ be a stratified dynamical ideal. In the associated permutation model, every set is either a countable union of finite sets or it contains an injective image of $\gw$.
\end{theorem}

\noindent In particular, the permutation model contains no amorphous sets. A countable union of finite sets can certainly be an infinite, Dedekind finite set. However, such an eventuality can be excluded if, for example, the ideals in the stratification of $I$ are simple. Then, even the union is simple (Example~\ref{unionexample}), therefore in the permutation model a countable union of finite sets must be countable by Theorem~\ref{simpletheorem}.

\begin{proof}
Let $\{I_n\colon n\in\gw\}$ be the stratification of $I$. Let $A$ be a set in the permutation model. Let $a\in I$ be a set such that $\pstab(a)\subseteq\stab(A)$ holds, and pick $n\in\gw$ such that $a\in I_n$ holds. For every number $m\in\gw$ let $A_m=\{B\in A\colon\exists b\in I_m\ \pstab(b)\subseteq\stab(B)\}$. By the invariance of supports (Proposition~\ref{siproposition}), the sequence $\langle A_m\colon m\in\gw\rangle$ belongs to the permutation model, and $A$ is the increasing union of the sets $A_m$. It will be enough to show that if one of the sets $A_m$ is infinite, then it contains an injective image of $\gw$.

Suppose then that $m\geq n$ is such that $A_m$ is infinite. Let $B_k$ for $k\in\gw$ be some injective $k$-tuples of elements of $A_m$. Note that since $I_m$ is an ideal, for each $k\in\gw$ there is a set $b_k\in I_m$ such that $\pstab(b_k)\subseteq\stab(B_k)$. Use the stratification assumption to find group elements $\gamma_k\in\pstab(a)$ such that $c=\bigcup_{k\in\gw}\gamma_k\cdot b_k\in I_{m+1}$.  Consider the sequence $C=\langle \gamma_k\cdot B_k\colon k\in\gw\rangle$. By the invariance of supports (Proposition~\ref{siproposition}), $\pstab(c)\subseteq\stab(C)$, so $C$ belongs to the permutation model. Since each $\gamma_k$ fixes the set $A$, the sequence $C$ consists of $k$-tuples of elements of $A$ for each $k$. Clearly, the set $\bigcup_k\rng(\gamma_k\cdot B_k)$ is an infinite countable subset of $A$ as required.
\end{proof}

\begin{example}
Let $X$ be a countable dense linear ordering without endpoints. Let $\Gamma$ be its automorphism group acting on $X$ by application. Let $\ga>\gw$ be a countable ordinal closed under ordinal addition, and let $I_\ga$ be the ideal of subsets of $X$ which are well-ordered of ordertype less than $\ga$. If $\ga$ is closed under ordinal multiplication, then $\Gamma\acts X, I_\ga$ is a stratified dynamical ideal. The ideal is also simple by Example~\ref{ordersimpleexample}; in conclusion, in the associated permutation model every infinite set contains an injective image of $\gw$.
\end{example}

\begin{proof}
Let $\langle\gb_n\colon n\in\gw\rangle$ be an increasing sequence of ordinals closed under ordinal addition, converging to $\ga$, and such that $\gb_{n+1}>(\gw\cdot\gb_n)\cdot\gb_n$. Then $I_\ga$ is the increasing union of $I_{\gb_n}$ for $n\in\gw$. In addition, an inspection of the proof of Example~\ref{wocompleteexample} shows that Definition~\ref{stratifieddefinition} holds for this stratification of $I_\ga$.
\end{proof}

\begin{example}
Let $\kappa$ be an uncountable cardinal, let $X$ be a set of cardinality at least $\kappa$, let $\Gamma$ be the group of all permutations of $X$ acting by application, and let $I_\kappa$ be the ideal of subsets of $X$ of cardinality smaller than $\kappa$. If $\kappa$ has countable cofinality, then $\Gamma\acts X, I$ is a stratified dynamical ideal. The ideal is simple by Example~\ref{cardinalitysimpleexample}; in conclusion, in the associated permutation model every infinite set contains an injective image of $\gw$.
\end{example}

\begin{proof}
Just let $\langle \lambda_n\colon n\in\gw\rangle$ be an increasing sequence of cardinals converging to $\kappa$. Then $I_\kappa$ is an increasing union of ideals $I_{\lambda_n^+}$ for $n\in\gw$. Each of these dynamical ideals is $\gs$-complete, and in fact has cofinal orbits by Example~\ref{cardinalitycofinalexample}; thus, Definition~\ref{stratifieddefinition} holds for this stratification of $I_\kappa$.
\end{proof}

\begin{example}
In the fairly popular situation where $I$ is generated by a countable increasing sequence of invariant sets, $\Gamma\acts X, I$ is a stratified dynamical ideal. For example, when $\mathbb{Z}$ acts on a countable set $X$ and $I$ is the ideal of finite sets, then there are two cases. Either there is an infinite orbit. In such a case, for any $x\in X$ in the infinite orbit, $\stab(x)=\{0\}$ and $W[[X]]=V[[X]]$; in this trivial case, the full axiom of choice holds in the associated permutation model. Or, all orbits are finite, and then $X$ is an increasing union of countably many finite invariant sets.
\end{example}

\bibliographystyle{plain} 
\bibliography{odkazy,zapletal,shelah}

\def\germ{\frak} \def\scr{\cal} \ifx\documentclass\undefinedcs
  \def\bf{\fam\bffam\tenbf}\def\rm{\fam0\tenrm}\fi 
  \def\defaultdefine#1#2{\expandafter\ifx\csname#1\endcsname\relax
  \expandafter\def\csname#1\endcsname{#2}\fi} \defaultdefine{Bbb}{\bf}
  \defaultdefine{frak}{\bf} \defaultdefine{=}{\B} 
  \defaultdefine{mathfrak}{\frak} \defaultdefine{mathbb}{\bf}
  \defaultdefine{mathcal}{\cal}
  \defaultdefine{beth}{BETH}\defaultdefine{cal}{\bf} \def\bbfI{{\Bbb I}}
  \def\mbox{\hbox} \def\text{\hbox} \def\om{\omega} \def\Cal#1{{\bf #1}}
  \def\pcf{pcf} \defaultdefine{cf}{cf} \defaultdefine{reals}{{\Bbb R}}
  \defaultdefine{real}{{\Bbb R}} \def\restriction{{|}} \def\club{CLUB}
  \def\w{\omega} \def\exist{\exists} \def\se{{\germ se}} \def\bb{{\bf b}}
  \def\equivalence{\equiv} \let\lt< \let\gt>
\begin{thebibliography}{1}

\bibitem{howard:ac}
Paul Howard and Jean~E. Rubin.
\newblock {\em Consequences of the Axiom of Choice}.
\newblock Math. Surveys and Monographs. Amer. Math. Society, Providence, 1998.

\bibitem{jech:cardinals}
Thomas Jech.
\newblock On cardinals and their successors.
\newblock {\em Bull. Acad. Polon. Sci. S{\' e}r. Sci. Math. Astronom. Phys.},
  14:533--537, 1966.

\bibitem{jech:choice}
Thomas Jech.
\newblock {\em The Axiom of Choice}.
\newblock North-Holland, New York, 1973.

\bibitem{jech:newset}
Thomas Jech.
\newblock {\em Set Theory}.
\newblock Springer Verlag, New York, 2002.

\bibitem{karagila:dc}
Asaf Karagila and Jonathan Schilhan.
\newblock Geometric condition for dependent choice.
\newblock {\em Acta Math. Hungar.}, 172:34--41, 2024.

\bibitem{kirby:annulus}
R.~C. Kirby.
\newblock Stable homeomorphisms and the annulus conjecture.
\newblock {\em Ann. Math.}, 89:575--582, 1969.

\bibitem{paolini:index}
Gianluca Paolini and Saharon Shelah.
\newblock The strong small index property for free homogeneous structures.
\newblock In {\em Research Trends in Contemporary Logic}. College Publications,
  London, 2020.

\bibitem{young:personal}
Justin Young.
\newblock Ph. {D}. {T}hesis.
\newblock 2026.

\end{thebibliography}

\end{document}